\documentclass[leqno]{amsart} 
\usepackage[utf8]{inputenc}

\usepackage{url}

\usepackage{amsthm, amsfonts, amssymb, amsmath, nccmath}

\usepackage{enumerate}
\usepackage{calrsfs}
\usepackage{mathtools}

\usepackage{graphicx}
\usepackage{color}
\usepackage{tikz-cd}

\newcommand{\sen}{\,{\rm sin}}

\newcommand{\wt}{\widetilde}

\newcommand{\wh}{\widehat}
\newcommand{\m}{{\rm m}}

\newcommand{\ov}[1]{{\overline{#1}}}

\def\prod#1{\langle{#1}\rangle}

\newtheorem{theorem}{Theorem}[section]
\newtheorem{lemma}[theorem]{Lemma}
\newtheorem{proposition}[theorem]{Proposition}
\newtheorem{corollary}[theorem]{Corollary}
\theoremstyle{definition}
\newtheorem{example}[theorem]{Example}
\newtheorem{remark}[theorem]{Remark}

\newtheorem{theoremA}{Theorem}

\renewcommand{\qed}{\hfill \nobreak \ifvmode \relax \else
      \ifdim\lastskip<1.5em \hskip-\lastskip
      \hskip1.5em plus0em minus0.5em \fi \nobreak
      \vrule height0.75em width0.5em depth0.25em\fi}

\renewenvironment{proof}[1][Proof:]{\begin{trivlist}
\item[\hskip \labelsep {\bfseries #1}]}{\qed \end{trivlist}}

\newcommand{\Ad}{\mathrm{Ad}}
\newcommand{\g}{\mathfrak{g}}
\renewcommand{\k}{\mathfrak{k}}

\renewcommand{\a}{\mathfrak{a}}

\renewcommand{\m}{\mathfrak{m}}
\newcommand{\s}{\mathfrak{s}}
\newcommand{\z}{\mathfrak{z}}
\renewcommand{\u}{\mathfrak{u}}

\newcommand{\gl}{\mathfrak{gl}}

\newcommand{\so}{\mathfrak{so}}

\renewcommand{\t}{\mathfrak{t}}

\newcommand{\R}{\mathbb{R}}
\newcommand{\C}{\mathbb{C}}
\newcommand{\Z}{\mathbb{Z}}

\newcommand{\ad}{\mathrm{ad}}

\newcommand{\SO}{\mathrm{SO}}

\title{Counting geodesics on compact symmetric spaces}

\author[L. Seco]{Lucas Seco$^*$}
\author[M. Patr\~ao]{Mauro Patr\~{a}o$^*$}
\address{$^*$Universidade de Brasília -- Departamento de Matemática, Campus Darcy Ribeiro, Asa Norte 70910-900 - Brasília, DF - Brazil}
\email{lseco@unb.br, mpatrao@mat.unb.br}

\begin{document}

\maketitle

\begin{abstract}
We describe the inverse image of the Riemannian exponential map at a basepoint of a compact symmetric space as the disjoint union of so called focal orbits through a maximal torus. These are orbits of a subgroup of the isotropy group acting in the tangent space at the basepoint. We show how their dimensions (infinitesimal data) and connected components (topological data) are encoded in the diagram, multiplicities, Weyl group and lattice of the symmetric space. Obtaining this data is precisely what we mean by counting geodesics. This extends previous results on compact Lie groups.  We apply our results to give short independent proofs of known results on the cut and conjugate loci of compact symmetric spaces.

\end{abstract}

\bigbreak
	
\bigbreak

\section{Introduction}
The set of smooth geodesics connecting two points $p$ and $q$ of a Riemannian manifold $M$ can be viewed as the tangent directions at $p$ in the inverse image of $q$ by the Riemannian exponential map $\exp_p: T_p M \to M$ at the basepoint $p$. To describe this set is roughly what we mean by {\em counting} geodesics.

We can use the action of the isometry group of $M$ to move around a geodesic $\gamma$ connecting $p$ and $q$, producing more geodesics connecting them. Indeed, let $K$ be isotropy subgroup of isometries that fix $p$ and let $K^q$ be the subgroup of $K$ of isometries that further fix $q$. Then the whole orbit $K^q \gamma$ consists of geodesics between $p$ and $q$.  This orbit can also be seen in the tangent space $T_pM$ to $p$ if we identify the geodesic $\gamma$ with its initial velocity $H = \gamma'(0)$. Then $\gamma(t) = \exp_p(tH)$, $q = \gamma(1) = \exp_p(H)$ and, for $k \in K$, $k \gamma(t) = \exp_p( t k H )$, where we denote by $k H = d_p k H$ the isotropy action of $K$ in $T_pM$. 
The orbit $K^q \gamma$ then gets identified with so called {\em focal orbit} of $H$
\begin{equation}
\label{eq-def-focal-orbit}
\mathcal{F}(H) = K^q H \subseteq T_pM
\end{equation}
It follows that $K^q$ acts in the inverse image $\exp_p^{-1}(q)$ which is, thus, a union of focal orbits. Describe this union and count the dimension and connected components of each focal orbit is precisely what we mean by {\em counting} geodesics.


In this article we count the geodesics of compact Riemannian symmetric spaces $S = U/K$: they are natural candidates for that, since their symmetry properties allows us to study their isotropy orbits in terms of Lie algebraic data. We now state our main results. 
Let $\t \subseteq T_pM$ be a maximal flat torus with corresponding lattice 
\begin{equation}
\label{def:lattice}
\Gamma = \{ \gamma \in \t: \exp_p(\gamma) = p \}
\end{equation}
restricted roots, Stiefel diagram $\mathcal{S}$, root hyperplanes and Weyl group $W$ (see Section \ref{sec:prelim}).
Let $q = \exp_p(H)$ with $H \in T_pM$, using the isotropy action of $K$ in $T_pM$ we can assume that $H \in \t$. Consider in $W$ the centralizer of $H\mbox{ mod }\Gamma$ given by
\begin{equation}
\label{def:centralizer}
W^q = \{ w \in W: wH \in H + \Gamma \}
\end{equation}
and the normal subgroup $W^q_0$ generated by reflections around root hyperplanes paralell to the hyperplanes in $\mathcal{S}$ which cross $H$.

%
 
\begin{theoremA}
\label{thm:main}
The geodesics between $p$ and $q$ of the compact symmetric space $S = U/K$ are given by the disjoint union of focal orbits through $H + \Gamma$
$$
\exp_p^{-1}(q) = \bigcup_{\gamma \in \Gamma} \mathcal{F}(H + \gamma)
$$
where:
\begin{enumerate}
    \item $\dim \mathcal{F}(H + \gamma) = $ equals to the number of non-root hyperplanes of $\mathcal{S}$ that cross $H + \gamma$, counted with multiplicity.
    
    \item $\mathcal{F}(H + \gamma)$ intersects $\t$ in the orbit $W^q(H+\gamma)$, its connected components are diffeomorphic to one another, correspond to the left $W^q_0$-orbits in $W^q (H+\gamma)$ and are in bijection with the quotient group $W^q/W^q_0$.
    
    \item Each connected component of $\mathcal{F}(H + \gamma)$ corresponds to a unique homotopy class of geodesics between $p$ and $q$.
\end{enumerate}
\end{theoremA}

This result shows how infinitesimal (dimension) and topological (connected components) data of the inverse image of the exponential of $S$ can be read from data in a maximal flat torus of $S$, showing how we can count geodesics just by looking at $\t$.  Note that the lattice $\Gamma$ is topological data of $S=U/K$ (see Theorem \ref{thm:pi1-introd}), while the roots, their multiplicity, Stiefel diagram and Weyl group is infinitesimal data and depends only on $(\u, \k)$.


This extends previous results of \cite{SS18}, where the subgroup \eqref{def:centralizer} was introduced to count the geodesics of compact Lie groups (see Example \ref{ex:grupos})
%
Note that, since focal orbits of $S$ can be disconnected (see Section \ref{sec:examples}), but are always connected for simply connected $S$ (see Theorem \ref{thm:simply-connected-introd}), we cannot obtain our general result by projecting it from the result on the simply connected covering. 
For the same reason, since only simply connected $S$ is guaranteed to embed in its isometry group $U$, we cannot obtain the general result from the compact group case. Note, furthermore, that the universal cover of $S$ is not compact when $S$ has an euclidean factor.

\begin{remark}
\label{rmk:focal-equivalents}
If one wishes to obtain all the geodesics of a fixed length $|H|$ between $p$ and $q$, one first obtains the so called focal equivalents of $H$ on $\t$, i.e.\ those $\Gamma$-equivalents $H + \gamma$ with length $|H + \gamma| = |H|$, then considers the focal orbits through them. The focal equivalents can be obtained from the lattice $\Gamma$ by a geometric construction involving mediator hyperplanes (see Section 4 of \cite{SS18}).
\end{remark}

We apply Theorem \ref{thm:main} to extends to compact symmetric spaces $S$ the following result, which is well known for symmetric spaces of compact type (see Proposition VI.2.4 of \cite{loos}). 

\begin{theoremA}
\label{thm:pi1-introd}
Let $\Gamma_0$ be the fundamental lattice. Then $\pi_1(S) = \Gamma/\Gamma_0$.
\end{theoremA}

We use this to prove the following connectedness result.

\begin{theoremA}
\label{thm:simply-connected-introd}
Let $S$ be simply connected. Then $\Gamma = \Gamma_0$ and $W^q = W^q_0$. In particular, $K^q$ is connected and thus all focal orbits are connected.
\end{theoremA}

We then apply our results to provide short independent proofs of classical results by Crittenden \cite{crittenden} and Sakai \cite{sakai} that characterize, respectively, the first conjugate locus and the cut locus of a compact symmetric space. 








\bigbreak

The structure of the article is as follows.
In Section 2 we introduce the notation and standard results of our setup.
In Section 3 we extend to a maximal flat torus of a compact symmetric space a diagonalization result about the action of the Weyl group.
In Section 4 we prove Theorem \ref{thm:main}.
In Section 5 we prove Theorem \ref{thm:pi1-introd}.
In Sections 6 we apply our results to provide independent proofs of classical results by Crittenden and Sakai.

In a forthcoming paper we will use the results of this article to refine the results of Bott \cite{bott} about the space of minimimal geodesics of $S$ and provide a related characterization of symmetric flag manifolds.

\bigbreak

We would like to thank K-H.\ Neeb and C.\ Gorodski for helpful conversations about symmetric spaces.

\subsection{Examples}\label{sec:examples}
We finish this introduction by illustrating our main results in examples of low rank.

A rank 0 compact symmetric space is a torus $T$, for which it is well known that $\exp_p^{-1}(q) = \bigcup_{\gamma \in \Gamma} (H + \gamma)$ and $\pi_1(T) = \Gamma$. These become special cases of Theorems \ref{thm:main} and \ref{thm:pi1-introd}, since the Stiefel diagram, fundamental lattice and Weyl group of a torus are trivial.

\begin{example}\label{ex:rank1}
(rank 1)
The round $2$-sphere $S^2 \subseteq \R^3$, has exponential map $\exp_p(H) = \cos(|H|)p + \sen(|H|)H/|H|$ for $H \neq 0$. The inverse image $\exp_p^{-1}(q)$ is a concentric union of circles of radii $2k\pi $ if $q = p$, a concentric union of circles of radii $(2k+1)\pi$ if $q = -p$, otherwise it is the discrete set $H + 2\pi k H/|H|$ along the direction of $H$, where $\exp_p(H) = q$, $k \in \Z$. 

This is recovered by Theorem \ref{thm:main} as follows.
$S^2 = \SO(3)/K$ is a symmetric space, where the isotropy of $p=(0,0,1)$ is $K=\left\{
\begin{psmallmatrix}
\ast & \ast & \\
 \ast & \ast & \\
 &    & 1
\end{psmallmatrix}
\in \SO(3)
\right\} = 
\SO(2)$, which is an open subgroup of the fixed point set 
of the conjugation of $\SO(3)$ by the reflection on the $z$-axis $\begin{psmallmatrix}
1 &  & \\
 & 1 & \\
 &    & -1
\end{psmallmatrix}$, given by $\left\{
\begin{psmallmatrix}
\ast & \ast & \\
 \ast & \ast & \\
 &    & \pm 1
\end{psmallmatrix}
\in \SO(3)
\right\}$. The tangent plane $T_p S^2$ is the $(x,y)$-plane with the isotropy action of $\SO(2)$, where the $x$ axis is a maximal flat torus $\t$ with simple root $\alpha(x) = x$ with multiplicity $1$. Thus we have Stiefel diagram $\mathcal{S} = \pi \Z$, whose nonzero points are conjugate tangent points along $x$, Weyl group $W = \{\pm 1\}$ and lattice $\Gamma = 2\pi \Z$.


Let $q = \exp_p(H)$.
Then $q = p$ iff $H \in \Gamma$, in this case $K^q = K$ so that focal orbits through non-zero tangent vectors are circles and $\exp_p^{-1}(q)$ is the union of concentric circles centered on the origin passing through $\Gamma$, thus circles of radii $2k\pi$.  
Furthermore $q = -p$ iff $H \in \pi + \Gamma$, in this case once again $K^q = K$ so that $\exp_p^{-1}(q)$ is the union of circles centered on the origin passing through $\pi + \Gamma$, thus circles of radii $(2k + 1)\pi$.
Note that $\mathcal{S} = \Gamma \cup ( \pi + \Gamma)$ so that in both previous cases all focal orbits cross the diagram, thus $W^q_0 = \{ \pm 1 \}$.
Furthermore, $W^q = \{ \pm 1 \}$ so that the focal orbit through $H + \gamma$ meets $\t$ in $\pm(H + \gamma)$ and is connected. Finally, if $q \neq \pm p$, then $K^q = 1$ so that the focal orbits are singletons and $\exp_p^{-1}(q)$ is the discrete set $H + \Gamma$ in the $x$ axis. Note that in this case no focal orbit crosses the diagram and $W^q = W^q_0 = 1$.  In all cases both $K^q$ and (thus) all focal orbits are connected.

We can get disconnected focal orbits on the real projective $2$-space $\R P^2$. Let $\pi: S^2 \to \R P^2$ be the canonical projection $\pi(v) = [v] = \pm v$, then the exponential at $[p] = (0,0,\pm 1)$ is $\exp_{[p]} = \pi \circ \exp_p$.
Let $[q] = (\pm 1,0,0)$, so that $[q] = \exp_{[p]}(H)$ for $H = \pi/2 \in \t$, 
then $\exp_{[p]}^{-1}([q])$ is the union of $\pi/2 + k\pi$, $k \in \Z$.
This is recovered by Theorem A as follows. $\R P^2 = \SO(3)/K$ is a symmetric space, where the isotropy of $[p]$ is $K = \left\{
\begin{psmallmatrix}
\ast & \ast & \\
 \ast & \ast & \\
 &    & \pm 1
\end{psmallmatrix}
\in \SO(3)
\right\} = \SO(2) \cup J\, \SO(2)$,
where $J =
\begin{psmallmatrix}
1 &  & \\
 & -1 & \\
 &    & -1
\end{psmallmatrix}$ restricts to a reflection in the $xy$-plane.
Since $\exp_{[p]} = \pi \circ \exp_p$, the infinitesimal setup is the same as the $2$-sphere, now with lattice $\Gamma = \pi\Z$. 
We have that
$K^{[q]} = \left\{
\begin{psmallmatrix}
\pm 1 &  & \\
 & \pm 1 & \\
 &    & \pm 1
\end{psmallmatrix}
\in \SO(3)
\right\}$ is disconnected with 4 components. 
%
%
It follows that $\exp_{[p]}^{-1}([q])$ is the union focal orbits $\pm (\pi/2 + k\pi)$, $k \in \Z$, where each focal orbit has 2 components.
By item (2) of Theorem A, these connected components can also be obtained from the corresponding Weyl groups: since $\pi/2$ does not cross the diagram $\mathcal{S} = \pi \Z$, we have $W^{[q]}_0 = 1$, and since $-\pi/2 = \pi/2 - \pi$, it follows that $-1$ centralizes $\pi/2$ modulo the lattice $\Gamma = \pi\Z$, so that $W^{[q]} = \{\pm 1\}$.
\end{example}

\begin{example}(rank 2)
\label{ex:posto 2}
Let $S = \mathrm{Gr}_2(\R^4)$ be the Grassmanian of planes in $\R^4$, denote by $[u,v] \in S$ the plane spanned by two l.i.\ vectors $u,v \in \R^4$ and by $e_1, e_2, e_3 , e_4$ the canonical basis of $\R^4$.
$S$ is a symmetric space of $\SO(4)$ with the canonical biinvariant metric and orthogonal involution given by conjugation by $I_{2,2} = \begin{psmallmatrix}
I & 0 \\
0 & -I \\
\end{psmallmatrix}$, which has fixed point set $K = \mathrm{S}(\mathrm{O}(2)\times\mathrm{O}(2))$, precisely the isotropy of $p = [e_1, e_2]$ in $\SO(4)$.

We have that $\k = \left\{ \begin{psmallmatrix}
A & 0 \\
0 & B \\
\end{psmallmatrix}:\, A, B \in \so(2) \right\}$, 
thus its orthocomplement is
$\s = \left\{ \begin{psmallmatrix}
0 & -C^t \\
C & 0 \\
\end{psmallmatrix}:\, C \in \gl(2) \right\}$, 
with maximal torus 
$\t = \left\{ \begin{psmallmatrix}
0 & -D^t \\
D & 0 \\
\end{psmallmatrix}:\, D = \begin{psmallmatrix} \theta_1 & \\ & \theta_2 \\ \end{psmallmatrix} \right\}$
depicted below as vectors $(\theta_1, \, \theta_2)$.
\begin{center}
    \def\svgwidth{10cm}
\begingroup%
  \makeatletter%
  \providecommand\color[2][]{%
    \errmessage{(Inkscape) Color is used for the text in Inkscape, but the package 'color.sty' is not loaded}%
    \renewcommand\color[2][]{}%
  }%
  \providecommand\transparent[1]{%
    \errmessage{(Inkscape) Transparency is used (non-zero) for the text in Inkscape, but the package 'transparent.sty' is not loaded}%
    \renewcommand\transparent[1]{}%
  }%
  \providecommand\rotatebox[2]{#2}%
  \newcommand*\fsize{\dimexpr\f@size pt\relax}%
  \newcommand*\lineheight[1]{\fontsize{\fsize}{#1\fsize}\selectfont}%
  \ifx\svgwidth\undefined%
    \setlength{\unitlength}{292.32583534bp}%
    \ifx\svgscale\undefined%
      \relax%
    \else%
      \setlength{\unitlength}{\unitlength * \real{\svgscale}}%
    \fi%
  \else%
    \setlength{\unitlength}{\svgwidth}%
  \fi%
  \global\let\svgwidth\undefined%
  \global\let\svgscale\undefined%
  \makeatother%
  \begin{picture}(1,0.65426884)%
    \lineheight{1}%
    \setlength\tabcolsep{0pt}%
    \put(0,0){\includegraphics[width=\unitlength,page=1]{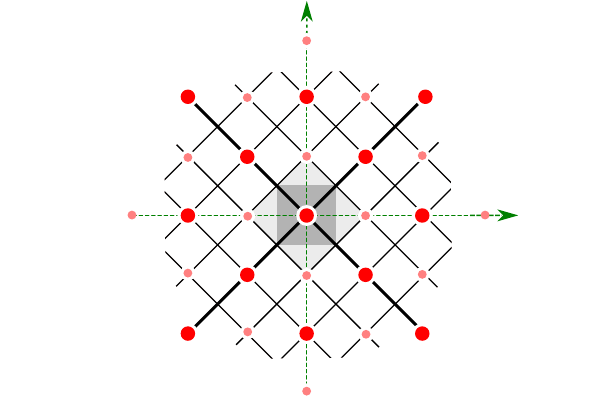}}%
    \put(0.81825213,0.33066534){\color[rgb]{0,0.50196078,0}\makebox(0,0)[lt]{\lineheight{1.25}\smash{\begin{tabular}[t]{l}$\theta_1$\end{tabular}}}}%
    \put(0.52063892,0.62827861){\color[rgb]{0,0.50196078,0}\makebox(0,0)[lt]{\lineheight{1.25}\smash{\begin{tabular}[t]{l}$\theta_2$\end{tabular}}}}%
    \put(0.56649205,0.33496947){\color[rgb]{0,0,1}\makebox(0,0)[lt]{\lineheight{1.25}\smash{\begin{tabular}[t]{l}$H$\end{tabular}}}}%
    \put(0.40742294,0.26188549){\color[rgb]{0,0,1}\makebox(0,0)[lt]{\lineheight{1.25}\smash{\begin{tabular}[t]{l}$H''$\end{tabular}}}}%
    \put(0,0){\includegraphics[width=\unitlength,page=2]{grassmann-lattice.pdf}}%
  \end{picture}%
\endgroup%

\end{center}
The simple restricted roots are $\alpha = \theta_1 - \theta_2$, $\beta = \theta_1 + \theta_2$, each with mutiplicity $1$ and, since these roots are orthogonal, this is a symmetric space of type $A_1 \times A_1$. The corresponding coroot vectors are $H^\vee_\alpha = (1,-1)$, $H^\vee_\beta = (1,1)$ with fundamental lattice $\Gamma_0 = \{ (\pi k, \pi l) \in \Gamma:$ $k,l \in \Z$ such that $k+l$ is even$\}$, depicted by the larger dots above.
The Stiefel diagram is then given by $\theta_1 - \theta_2$, $\theta_1 + \theta_2 \in \pi\Z$, depicted by the slanted lines above, where the root hyperplanes are depicted as thicker lines. Since $\exp_p(X) = e^X [e_1,e_2]$, for $X \in \s$, 
where
$$
\setlength\arraycolsep{0.1pt}
\begin{matrix}
  &
\begin{psmallmatrix}
 & & -\theta_1 & \\ 
 & & & -\theta_2 \\ 
\theta_1 & & & \\ 
 & \theta_2 & & \\
\end{psmallmatrix} \\
e &   \\
\end{matrix}
=
\begin{psmallmatrix}
\cos(\theta_1) & \hspace{10pt}  & -\sin(\theta_1) & \hspace{10pt} \\ 
 & 1 &  &  \\ 
\sin(\theta_1) & \hspace{10pt} & \cos(\theta_1) &  \hspace{10pt} \\
 &  &  & 1 \\ 
\end{psmallmatrix}
\begin{psmallmatrix}
 1 &  &  &  \\ 
\hspace{10pt}  & \cos(\theta_2) & \hspace{10pt}  & -\sin(\theta_2) \\ 
 &  & 1 &  \\ 
\hspace{10pt}  & \sin(\theta_2) & \hspace{10pt} & \cos(\theta_2) \\
\end{psmallmatrix}
$$
it follows that the lattice of $S$ is $\Gamma = \pi( \Z \times \Z)$, depicted by the dots above.
The Weyl group $W$ is given by either the symmetries of the dark grey square above, with vertices $(\pm \frac{\pi}{2}, \pm \frac{\pi}{2})$, which is the closure of the Dirichilet domain $\mathcal{D}$ of $\Gamma$.
By Theorem \ref{thm:main} it follows that $\pi_1(S) = \Gamma/\Gamma_0 = \Z_2$.

Let $q = [e_3,e_4]$ and note that $q = \exp_p(H)$, for $H = (\frac{\pi}{2},\, \frac{\pi}{2}) \in \t$. By inspecting the diagram, it follows that $W^q_0 = W$, which implies that $W^q = W$, so that all focal orbits $\mathcal{F}(H+\gamma)$ are connected.  
Furthermore, note that each $H + \gamma$ crosses either one or two nonroot hyperplanes, thus  $\mathcal{F}(H+\gamma)$ is connected with dimension 1 or 2. Note that $K^q = K$, thus the focal orbits $\mathcal{F}(H+\gamma)$ are flag manifolds of $S$. Since $S$ is of type $A_1 \times A_1$, it follows that $\mathcal{F}(H+\gamma)$ is a product of projective lines whenever both $\alpha(H+\gamma)$, $\beta(H+\gamma)$ are not zero, and a projective line if one of them is zero. It follows that $\exp_p^{-1}(q)$ is a disjoint union of 1-spheres passing through the non-regular elements of $H + \Gamma$ and tori passing through the regular elements of $H + \Gamma$. Since $\gamma = (\pi k,\, \pi l)$, $k,l \in \Z$, note that $H + \gamma = (\frac{\pi}{2}(2k + 1),\, \frac{\pi}{2}(2l + 1))$. 

Note that $H$ lies in the boundary of $\mathcal{D}$ and that among the $\Gamma$-equivalents to $H$ with the same norm, 
both $H$ and $r_\beta(H) = -H$ lie in the focal orbit $\mathcal{F}(H)$, while both $H' = (-\frac{\pi}{2},\, \frac{\pi}{2})$ and $r_\alpha(H') = (\frac{\pi}{2},\, -\frac{\pi}{2})$ do not. Thus the minimal geodesics from $p$ to $q$ in $S$ are given by the disjoint union of $\mathcal{F}(H)$ and $\mathcal{F}(H')$, a disjoint union of projective lines, 

Now let $q''=[e_3,e_2]$ and note that $q'' = \exp_p(H'')$, for $H'' = (-\frac{\pi}{2},0) \in \t$. By inspecting the diagram, it follows that $W_0^q = 1$ and $W^q = $ reflection around the $\theta_2$ axis, thus all focal orbits $\mathcal{F}(H''+\gamma)$ have two connected components. 
Furthermore, note that none $H'' + \gamma = (\pi k,\, \frac{\pi}{2}(2l + 1))$
crosses root hyperplanes, thus $\mathcal{F}(H''+\gamma) = (\pi k,\, \pm\frac{\pi}{2}(2l + 1))$. If follows that $\exp_p^{-1}(q'')$ is a disjoint union of $0$-spheres in $H'' + \Gamma$.

The universal cover of $S$ is the Grassmanian of oriented planes $\wt{S} = \wt{\mathrm{Gr}}_2(\R^4)$, which is a symmetric space with the same setup as above, except for the isotropy $K = \SO(2) \times \SO(2)$ of the oriented plane $p=\wt{[e_1,e_2]}$ and the lattice, which is now the fundamental lattice $\Gamma_0$. Now all focal orbits are connected (by 
Theorem \ref{thm:simply-connected-introd}),
with circles (or products) taking place of projective lines (or products).
Another way to obtain this is to use that $\wt{S}$ is isometric to product of $2$-spheres, since it is simply connected of type $A_1 \times A_1$, and use Example \ref{ex:rank1}.
Note that $H$ also lies in the boundary of the Dirichilet domain $\mathcal{D}_0$ of $\Gamma_0$, given by the light grey square above, with vertices $(\pm \pi, 0)$, $(0,\pm \pi)$.  
Now the $\Gamma_0$-equivalents to $H$ with the same norm are only 
$H$ and $r_\beta(H) = -H$, which lie in the focal orbit $\mathcal{F}(H)$. Thus, the minimal geodesics from $p$ to $q$ in $\wt{S}$ are given by the connected focal orbit $\mathcal{F}(H)$, a circle.

For $n > 4$, the Grassmannian $\mathrm{Gr}_2(\R^n)$ is a symmetric space of rank 2 with setup analogous to the above, except that now $\theta_1, \theta_2$ are also positive roots, so that it is irreducible of type $B_2$, with additional vertical and horizontal lines $\theta_1, \theta_2 \in \pi\Z$ in the above Stiefel diagram. 
In particular, its universal cover $\wt{\mathrm{Gr}}_2(\R^n)$ is not a product. 
%

\end{example}

\begin{example}
\label{ex:grupos}
Taking care to divide by 2, the above symmetric space results for $H/2$ generalize the compact Lie group results of \cite{SS18} for $H$.

Indeed, a compact connected Lie group $U$ with a biinvariant metric is a symmetric space of $U \times U$ acting on itself by left translations and acting on $U$ by left and right isometries $(g,h) \cdot u = g u h^{-1}$.
The isotropy $K$ of the basepoint $p=1 \in U$ is given by $gh^{-1} = 1$ so that $K = \{(u,u):~u\in U\}$ is the diagonal of $U$, which is the fixed point set of the involution $\sigma(g,h) = (h,g)$. Note that the action of the isotropy $K$ in $U$ is given by conjugation. Thus, the isotropy $K^h$ of $h \in U$ is the diagonal of the centralizer $U_h$ of $h$.
The equivariant diffeomorphism that maps $S = U \times U/K$ to $U$ intertwining the respective actions is given by the slope map $\pi((g,h)K) = gh^{-1}$. 

The $(-1)$-eigenspace of the involution on $\u \times \u$ is $\s = \{ (X, -X):\, X \in \u \}$. Let $\t$ be a maximal torus of $\u$, then $\a = \{ (H, -H):\, H \in \t \}$ is a maximal torus of $\s$. 
The differential of $\pi$ at $K$ maps $\s \to \u$, and is given by $$(X,-X) \mapsto 2X$$
It maps the roots of $S$ to the roots of $U$, the lattice $\Gamma$ of $S$ to the unit lattice $\Gamma$ of $U$, the diagram of $S$ to the diagram of $U$. It also maps the isotropy action of $K$ in $\s$ to the adjoint action of $U$ in $\u$, thus the Weyl group action of $S$ to the Weyl group action of $U$ and the isotropy action of $K^h$ in $\s$ to the adjoint action of the centralizer $U_h$ in $\u$. 
Finally, the focal orbit $K^h (H,-H) \subseteq \s$ gets mapped to the centralizer orbit $U_h (2H) \subseteq \u$.

Note that the unit quaternions $S^3$ as a Riemannian manifold with induced metric from $\R^4$ has $I(S^3)_0 = \mathrm{SO}(4)$. As a compact group with the biinvariant metric, thus as a symmetric space, it also has $I(S^3)_0 = \mathrm{SO}(4)$, since $S^3$ has an $S^3 \times S^3$ action by isometries, as above, thus a homomorphism $S^3 \times S^3 \to I(S^3)_0 = \mathrm{SO}(4)$ which is surjective since it is, in fact, the universal covering of $\mathrm{SO}(4)$.
\end{example}

\section{Preliminaries}\label{sec:prelim}

Let us recall the definitions, notations and basic results on compact symmetric spaces and their Lie algebras (see for example \cite{borel}, \cite{helgason} or \cite{sakai}). Since many of the standard results in the literature are for symmetric spaces of compact type, thus without euclidean factor, we take care to extend them to compact symmetric spaces, whenever needed.

A compact Riemannian symmetric pair $(U,K)$ consists of: a compact connected Lie group $U$, a closed subgroup $K$, a $U$-invariant Riemannian metric on $S = U/K$ and an involutive automorphism $\sigma$ of $U$ such that $K$ lies between $U^\sigma_0$ and $U^\sigma$, where $U^\sigma$ is the subgroup of fixed points of $\sigma$.
For short, we will call this a {\em compact symmetric space}. 
Denote by $\u$, $\k$ the Lie algebras of $U$, $K$, respectively, and again by $\sigma$ the involutive automorphism of $\u$ given by the differential at the identity of the automorphism of $U$.  Then $(\u, \k)$ is a compact orthogonal symmetric pair corresponding to $(U,K)$, where $\k = (+1)$-eigenspace of $\sigma$. Let $\s = (-1)$-eigenspace of $\sigma$, then $\u = \k \oplus \s$ is an orthogonal decomposition. For short, we will call this a {\em compact symmetric Lie algebra}.
We may identify the the tangent space $T_pS$ with $\s$ via the canonical projection $\pi: U \to S$, where $p = \pi(1)$ is the canonical basepoint.  

In the following we shall fix in $\u$ an $\Ad(U)$-invariant inner product. Such inner product is essentially unique in the semisimple factor of $\u$ and is arbitrary on the center of $\u$. We then restrict it to $\s$, obtaining an $\Ad(K)$-invariant inner product which extends uniquely to a $U$-invariant Riemannian metric on $S = U/K$.  Then the Riemannian exponential map of $S$ at $p$ is then given by
\begin{equation}
\label{eq-exp-riemanniana-S}
    \exp_p: \s \to S, \qquad X \mapsto \exp(X)p = \pi(\exp(X))
\end{equation}
where $\exp: \u \to U$ denotes the exponential mapping of the Lie group $U$. 
\subsection{Restricted roots}

Let $\t \subseteq \s$ be a maximal (flat) torus. They are all conjugate by the action of $\mathrm{Int}_\u(\k)$ in $\s$. To a real functional $\alpha \in \t^*$ corresponds the restricted root space
\begin{equation}
\label{def-restricted-root-space}
\u_\alpha = \{ X \in \u_\C: \, 
[H,X] = i \alpha(H) X,\, H \in \t 
\}
\end{equation}
where $i = \sqrt{-1}$ and we do not introduce the $2\pi$ factor used in \cite{sakai} and \cite{SS18}.
The (restricted) roots of $(\u,\k)$ are given by $\Pi = \{ \alpha \in \t^*: \, \u_\alpha \neq 0 \}$ with corresponding restricted root space decomposition of $\u_\C$ given by
\begin{equation}
\label{eq:centralizador-u}
\u_\C = \u_0 \oplus \sum_{\alpha \in \Pi} \u_\alpha
\end{equation}
The Stiefel diagram of the symmetric space is the union of the affine hyperplanes
\begin{equation}
\label{def-diagrama-esp-sim}
\mathcal{S} = \{ H \in \t: \, \alpha(H) \in \pi\Z,\text{ for some } \alpha \in \Pi \}
\end{equation}
where the root hyperplanes are the hyperplanes through the origin.

The set of regular elements $\t_{\mathrm{reg}}$ is the complement of the Stiefel diagram in $\t$.
A connected component of the complement of $\t_{\mathrm{reg}}$ is an alcove, and a connected component of the complement of the root hyperplanes is a Weyl chamber. The Weyl group acts simply transitively on the Weyl chambers. 

Fix a Weyl chamber $\t^+$ and corresponding positive roots $\Pi^+$ and simple roots $\Sigma \subseteq \Pi^+$.  For $\alpha \in \Pi^+$ let
$\k_\alpha = \k \cap (\u_\alpha \oplus \u_{-\alpha})$,
$\s_\alpha = \s \cap (\u_\alpha \oplus \u_{-\alpha})$, and let $\m = $ centralizer of $\t$ in $\k$, then
\begin{equation}
\label{eq:root-decompos-ks}
\k = \m \oplus \sum_{\alpha \in \Pi} \k_\alpha \qquad\qquad
\s = \t \oplus \sum_{\alpha \in \Pi} \s_\alpha
\end{equation}
where $\dim \k_\alpha = \dim \s_\alpha = m_\alpha$, the multiplicity of $\alpha$. Note that $\u_0 = \m \oplus \t$.



\subsection{Centralizers and the center}

Restricted roots of the compact symmetric Lie algebra $(\u, \k)$ are related to the roots of the compact Lie algebra $\u$ as follows. Let $\widetilde{\t}$ be maximal torus of $\u$ which contains $\t$.  To a real functional $\widetilde{\alpha} \in \widetilde{\t}^*$ corresponds the root space
$$
\widetilde{\u}_{\widetilde{\alpha}} = \{ X \in \u_\C: \, 
[H,X] =i \widetilde{\alpha}(H) X,\, H \in \widetilde{\t}
\}
$$
The roots of $\u$ are given by $\widetilde{\Pi} = \{ \widetilde{\alpha} \in \widetilde{\t}^*: \, \widetilde{\u}_{\widetilde{\alpha}} \neq 0 \}$ and we have the  root space decomposition
$$
\u_\C = \widetilde{\u}_0 \oplus \sum_{\widetilde{\alpha} \in \widetilde{\Pi}} \widetilde{\u}_{\widetilde{\alpha}}
$$
Since the adjoint representation of $\t$ in $\u_\C$ is the restriction of the adjoint representation of $\widetilde{\t}$, it follows that the restricted root spaces of $(\u, \k)$ split as a sum of roots spaces of $\u$ so that, for $\alpha \in \Pi$, we have
$$
\u_\alpha = \sum_{\widetilde{\alpha}|_\t = \alpha}
\widetilde{\u}_{\widetilde{\alpha}}
$$
Let $\widetilde{\Pi}_0 = \{ \widetilde{\alpha} \in \widetilde{\Pi}: \, \widetilde{\alpha}|_\t = 0 \}$, it follows that 
$\Pi = \{ \widetilde{\alpha}|_\t:\, \widetilde{\alpha} \in \widetilde{\Pi} - \widetilde{\Pi}_0 \}$.
Furthermore, a Weyl chamber of $\t$ can be chosen with corresponding positive roots $\widetilde{\Pi}^+$ so that $\Pi^+ = \{ \widetilde{\alpha}|_\t:\, \widetilde{\alpha} \in \widetilde{\Pi}^+ - \widetilde{\Pi}_0  \}$.

For $H \in \t$, we have that $\ad(H)$ is diagonal in $\u_\C$ with $0$-eigenspace $\widetilde{\u}_0$ and $i\alpha(H)$-eigenspace $\u_\alpha$, $\alpha \in \Pi$. For $h = \exp(H)$ we have that $\Ad(h) = e^{\ad(H)}$ is diagonal in $\u_\C$ with $1$-eigenspace $\widetilde{\u}_0$ and $e^{i\alpha(H)}$-eigenspace $\u_\alpha$, $\alpha \in \Pi$.
Thus, by \eqref{eq:root-decompos-ks} the centralizers 
$$
\k_H = \{ X \in \k: \ad(H)X = 0 \}
\subseteq
\k_h = \{ X \in \k: \Ad(h)X = X \}
$$
decompose as 
\begin{equation}
\label{eq-algebra-centralizador}
\k_H = \m \oplus \sum_{\alpha(H) = 0} \k_\alpha
\qquad
\k_h = \m \oplus \sum_{\alpha(H) \in 2\pi\Z} \k_\alpha
\end{equation}
where $\alpha$ ranges through the positive restricted roots.

Since $\u$ compact is reductive, it splits as the direct sum
\begin{equation}
\label{eq:centro}
\u = \g \oplus \z(\u)
\end{equation}
where $\g = [\u,\u]$ is compact semi-simple and $\z(\u)$ the center of $\u$. 
The compact symmetric Lie algebra $(\u, \k)$ is said to be of {\em compact type} when $\z(\u) = 0$. Note that both $\z(\u)$ and $\g$ are $\sigma$-invariant, since they are characteristic ideals.
Thus
\begin{eqnarray}
\g = (\k \cap \g) \oplus (\s \cap \g) \,\,\,\,\,\,& \qquad &
\z(\u) = (\k \cap \z(\u)) \oplus (\s \cap \z(\u)) \\
\label{eq:k-s-g}
\k = (\k \cap \g) \oplus (\k \cap \z(\u)) & \qquad &
\s = (\s \cap \g) \oplus (\s \cap \z(\u)) \\
&& \t = (\t \cap \g) \oplus (\t \cap \z(\u))
\end{eqnarray}

Let $\wh{\t} = \t \cap \g$. The rank of the compact symmetric Lie algebra $(\u, \k)$, hence of the symmetric space $U/K$, is $\dim \wh{\t}$. 
Restricting $\sigma$ to $\g$ we get a symmetric Lie algebra of compacty type with maximal torus $\wh{\t}$. Since the roots $\Pi$ of $\u$ annihilate $\t \cap \z(\u)$, it follows that the roots of $\g$ are the roots of $\u$ restricted to $\wh{\t}$. 
It also follows that $\mathcal{S} = (\mathcal{S} \cap \wh{\t}\, ) \oplus (\t \cap \z(\u))$ so that the diagram hyperplanes are, in fact, the sum of hyperplanes in $\wh{\t}$ with $\t \cap \z(\u)$.  

To each restricted root $\alpha$ corresponds the dual root vector $H_\alpha \in \wh{\t}$ such that $\prod{H_\alpha, H} = \alpha(H)$, for all $H \in \wh{\t}$. The corresponding coroot vector is $H^\vee_\alpha = 2H_\alpha/\langle{\alpha, \alpha}\rangle$, note that $\alpha(H^\vee_\alpha) = 2$. 





\begin{lemma}(Lemma 1.1 of \cite{sakai})
\label{lemma:sakai}
For $\widetilde{\alpha} \in \widetilde{\Pi}^+ - \widetilde{\Pi}_0$, there exist $X_{\widetilde{\alpha}} \in \k$, $Y_{\widetilde{\alpha}} \in \s$ such that:
\begin{enumerate}[(i)]
    \item $\{ X_{\widetilde{\alpha}}:\, \widetilde{\alpha}|_\t = \alpha \}$ is a basis of $\k_\alpha$, $\{ Y_{\widetilde{\alpha}}:\, \widetilde{\alpha}|_\t = \alpha \}$ is a basis of $\s_\alpha$, for $\alpha \in \Pi$.  
    
    Let $\widetilde{\alpha}|_\t = \alpha$, then 
    \item $[H,X_{\widetilde{\alpha}}] = \alpha(H) Y_{\widetilde{\alpha}}, \quad 
    [H,Y_{\widetilde{\alpha}}] = -\alpha(H) X_{\widetilde{\alpha}}, \quad$ for $H \in \t$.
    \item 
    $[X_{\widetilde{\alpha}},Y_{\widetilde{\alpha}}] = H^\vee_\alpha/2$
\end{enumerate}
\end{lemma}

In particular, since $\alpha(H^\vee_\alpha/2)=1$, to each positive restricted root $\alpha$ corresponds a subalgebra $\u(\alpha)$ of $\u$ spanned by $H^\vee_\alpha/2$, $Y_\alpha$, $X_\alpha$ with bracket relations
\begin{equation}
\label{eq:bracket-alpha}
[Y_\alpha, X_\alpha] = H^\vee_\alpha/2 \qquad 
[H^\vee_\alpha/2, Y_\alpha] = X_\alpha\qquad
[H^\vee_\alpha/2, X_\alpha] = Y_\alpha
\end{equation}

\subsection{Weyl group}\label{Weyl-group}

Consider the normalizer $M^*$ and the centralizer $M$ of $\t$ in $K$. The Weyl group of the symmetric space $U/K$ is then $W = M^*/M$. We can see an element $w \in W$ either as an element of $M^*$ acting by the adjoint action on $\t$ or as a linear isometry of $\t$. We have that $W$ also acts on the roots $\Pi$ by the coadjoint action $w^*\alpha(H)=\alpha(w^{-1}H)$.
By seeing an element of $W$ as an element of $M^*$ it follows at once that the lattice $\Gamma$ is invariant by the Weyl group. 
We have that $W$ is generated by simple reflections $r_\alpha$ around the root hyperplanes $\alpha=0$, where $\alpha$ varies in the simple roots and the reflection $r_\alpha$ of $\t$ around the hyperplane $\alpha = 0$ is given by $r_\alpha(H) = H - \alpha(H) H_\alpha^\vee$. Furthermore, the centralizer $W_H$  of $H$ in $W$ is generated by reflections $r_\alpha$ around the root hyperplanes which cross $H$, where $\alpha$ varies in the simple roots. In particular, if $H$ is regular then $W_H$ is trivial.

Regarding the maximal torus $\t$ as diagonal matrices, we have the following simultaneous diagonalization result: 
if $\t' \subseteq\s$ is a maximal torus, then there exists $k \in K_0$ such that $\Ad(k)\t' = \t$ (see Theorem II.2.3 of \cite{borel}).  In particular, 
\begin{equation}
\label{eq-diag}
\text{Given $X \in \s$ there exists $k \in K_0$ such that $\Ad(k)X \in \t$}
\end{equation}
Furthermore, for $H \in \t$ and $k \in K$, we have that 
\begin{equation}
\label{eq-lemma-well-known}
\text{If $\Ad(k) H \in \t$ then $\Ad(k)H = w H$ for some $w \in W$}
\end{equation}
which can be regarded as: a change of orthogonal base of a diagonal matrix which diagonalizes it again, can be attained by a permutation in the Weyl group (see Proposition VII.2.2 p.285 of \cite{helgason}, the proof applies equally to compact symetric spaces).

\subsection{Lattices}
\label{sec:lattices}

Consider the fundamental and the central lattices given, respectively, by
\begin{eqnarray}
\label{eq:central-fundamental-lattice}
    \Gamma_0 = \Z\text{-span } \{ \pi H^\vee_\alpha:\, \alpha \in \Pi \} \,\subseteq\,
    \Gamma_1 = \{ H \in \t:\, \alpha(H) \in \pi\Z,\, \alpha \in \Pi \}   
\end{eqnarray}
Actually, $\Gamma_0$ is a lattice of $\wh{\t}$. Furthermore, $\z(\u) \subseteq \Gamma_1$ so that $\Gamma_1$ may not be a lattice of $\t$, nevertheless, projecting onto $\wh{\t}$ parallel to $\t \cap \z(\u)$, we get 
\begin{eqnarray}
    \label{eq:lattice-1}
    \wh{\Gamma}_1 = \{ H \in \wh{\t}:\, \alpha(H) \in \pi\Z,\, \alpha \in \Pi \}   
\end{eqnarray}
which is a lattice of $\wh{\t}$, in fact is the lattice of the adjoint symmetric space $\wh{S}$ of $S$ obtained as follows. 

Note that the kernel of $\Ad: U \to \wh{U} = \Ad(U)$ is the center $Z(U)$, a characteristic subgroup, thus invariant by the involution $\sigma$, which then descends to the involution $\wh{\sigma}$ of $\wh{U}$ given by $\wh{\sigma}(\Ad(u)) = \Ad(\sigma(u))$.  Note that, since $\Ad(\exp(X)) = e^{\ad(X)}$ it follows that $\wh{\sigma}(e^{\ad(X)}) = e^{\ad(\sigma(X))}$. 
Let $\wh{K} = \wh{U}^{\wh{\sigma}}$ be the fixed point set of $\wh{\sigma}$, then $(\wh{U},\wh{K})$ is a compact Riemannian symmetric pair with correspoding symmetric space $\wh{S} = \wh{U}/\wh{K}$ and compact symmetric Lie algebra $\g$. Since $\Ad(K) \subseteq \wh{K}$, $\Ad$ descents to the projection
\begin{equation}
\label{eq:proj1}
S \to \wh{S} \qquad gK \mapsto \Ad(g)\wh{K}
\end{equation}
We have that $H \in \wh{\t}$ belongs to the lattice of $\wh{S}$ iff $e^{\ad(H)} \in \wh{K}$, thus iff
$$
\wh{\sigma}(e^{\ad(H)}) = e^{\ad(-H)} = e^{\ad(H)}
$$
thus iff $e^{\ad(2H)} = 1$. Since the eigenvalues of $e^{\ad(2H)}$ in $\g_\C$ are $1$ and $e^{2\alpha(H)i}$, for $\alpha$ varying in the roots of $\g$, it follows that 
the lattice of $\wh{S}$ is given by \eqref{eq:lattice-1}.

Recalling the lattice $\Gamma = \{ \gamma \in \t: \exp_p(\gamma) = p \}$ of $S$, the projection \eqref{eq:proj1} then implies that, projecting onto $\wh{\t}$ parallel to $\t \cap \z(\u)$, $\Gamma$ projects inside $\wh{\Gamma}_1$. Since the roots $\Pi$ of $\u$ annihilate $\t \cap \z(\u)$, it follows that $\Gamma \subseteq \Gamma_1$. Furthermore, we have the following.

\begin{proposition}
\label{propos:gamma-inv}
We have that $\Gamma_0 \subseteq \Gamma \subseteq \Gamma_1$. 
In particular, the Stiefel diagram $\mathcal{S}$ of $S$ is invariant by translations in the lattice $\Gamma$ of $S$, so that $\Gamma \subseteq \mathcal{S}$.
\end{proposition}
\begin{proof}
Let $U' = S^3$ be the unit sphere of the quaternions $\mathbb{H}$, whose Lie algebra is $\u' = \R i \oplus \R j \oplus \R k$ with bracket coming from the quaternionic product. Consider the orthogonal involution $\sigma'(q) = kqk^{-1}$, with fixed point subgroup $K' = $ unit circle of the $\R \oplus k\R$-plane, whose Lie algebra is $\k' = \R k$, complemented by $\s' = \R i \oplus \R j$. Take $H' = i/2$ in the maximal torus $\t' = \R i$ of $\s'$, whose exponential is the unit circle of the $\R \oplus i\R$-plane, thus its intersection with $K'$ is $\{\pm 1\}$ so that  $U'$ has lattice $\Gamma' = \pi \Z i$.

Let $Y' = j/2$, $X' = k/2$, which have bracket relations
$[Y', X'] = H'$,
$[H', Y'] = X'$,
$[H', X'] = Y'$.
By \eqref{eq:bracket-alpha}, to each positive root $\alpha$ corresponds a Lie algebra isomorphism $\u' \to \u(\alpha) \subseteq \u$ that maps  $H' \mapsto H^\vee_\alpha/2$, $Y' \mapsto Y_\alpha$, $X' \mapsto X_\alpha$, thus $i \mapsto H^\vee_\alpha$, $\t'$ into $\t$, $\s'$ into $\s$ and $\k'$ into $\k$. Since $U'$ is simply connected, this isomorphism can be integrated to a Lie group homomorphism $U' \to U$ that maps $e^{\theta i} \mapsto \exp(\theta H^\vee_\alpha)$ and $K'$ into $K$. Since $e^{\pi i} = -1 \in K'$, it follows that $\exp(\pi H^\vee_\alpha) \in K$ so that $\pi H^\vee_\alpha \in \Gamma$.
\end{proof}

For $H \in \t$, recall the Weyl subgroup $W^q$, centralizer of $H$ modulo the lattice $\Gamma$, and its subgroup $W^q_0$, generated by reflections around root hyperplanes paralell to the hyperplanes in $\mathcal{S}$ which cross $H$.
Since $W$ permutes the roots and $\Gamma \subseteq \mathcal{S}$, we have the following.

\begin{proposition}
$W_0^q$ is a normal subgroup of $W^q$. 
\end{proposition}

We will need some definitions and results from \cite{SS18}. Two vectors $X, Y \in \t$ are $\Gamma$-equivalent if $X - Y \in \Gamma$ and are called focal equivalents if, furthermore, $|X| = |Y|$.

\begin{lemma}[Lemma 4.2 of \cite{SS18}]
\label{lema:reticulado}
If $X \neq Y \in \t$ are focal equivalents then, for $\epsilon > 0$, $(1 + \epsilon)X \neq Y + \epsilon X$ are $\Gamma$-equivalent and such that 
$$|(1 + \epsilon) X| > |Y + \epsilon X|$$
\end{lemma}

The Dirichlet domain of $\Gamma$ is given by
\begin{equation}
\label{eq:def-dirichilet}
{\mathcal D} = \{ H \in \t: \, | H | < | H+\gamma |, \, \gamma \in \Gamma - 0 \}
\end{equation}
so that 
$$
\ov{{\mathcal D}} = \{ H \in \t: \, | H | \leq | H+\gamma |, \, \gamma \in \Gamma \}
$$
The next result shows that ${\mathcal D}$ is a fundamental domain of $\Gamma$ and characterizes its geometry.

\begin{proposition}[Proposition 4.3 of \cite{SS18}]
\label{prop:reticulado}
We have that
\begin{enumerate}[i)]
\item Given $H \in \t$, the minimal norm among $H + \Gamma$ is attained at $\ov{\mathcal D}$.

\item $\ov{\mathcal D}$ is nonempty, convex and compact.

\item ${\mathcal D}$ is the interior of $\ov{\mathcal D}$ so that $H \in {\mathcal D}$ iff $H$ is the unique element of minimal norm in $H + \Gamma$.  

\item Its boundary is $\partial {\mathcal D} = \{ H \in \ov{\mathcal D}:$ $H$ has nontrivial focal equivalents in $\t\}$.
\end{enumerate}
\end{proposition}

Now denote the alcove at the origin of the Stiefel diagram by
\begin{equation}
\label{eq:dirichilet-domain}
{\mathcal D}_0 = \{ H \in \t: \, \alpha(H) < \pi,\, \alpha \in \Pi \}
\end{equation}
so that 
\begin{equation}
\ov{{\mathcal D}_0} = \{ H \in \t: \, \alpha(H) \leq \pi,\, \alpha \in \Pi \}
\end{equation}

\begin{theorem}
\label{thm-steinberg}
If $\Gamma = \Gamma_0$ then 
\begin{enumerate}[i)]
    \item Its Dirichilet domain is ${\mathcal D}_0$,
    \item The focal equivalents of $H \in \ov{{\mathcal D}_0}$ in $\t$ are given by the orbit $W_0^q H$,
    \item $W^q = W^q_0$.
\end{enumerate}
\end{theorem}
\begin{proof}
Items (i) and (ii) are given by Theorem 4.5 of \cite{SS18}, here $\Gamma_0$ is $\pi \Z$-spanned by $H^\vee_\alpha$, instead of $\Z$-spanned.

For item (iii), let $w \in W^q$. Then $wH$ is a focal equivalent of $H$ so that, by the previous item, $wH = vH$, $v \in W^q_0$.  It follows that $v^{-1}w H = H$, so that $u = v^{-1}w$ is the product of reflections by simple roots $\alpha$ such that $\alpha(H) = 0$. Then $w = vu$ is the product of reflections by roots $\alpha$ such that $\alpha(H) \in \pi\Z$, so that $w \in W_h^0$.
\end{proof}

Note that this provides an alternative definition of $W^q_0$ as the centralizer of $H$ modulo the fundamental lattice $\Gamma_0$.

\subsection{Cartan covering}
Consider the following involution of $U$
\begin{equation}
u^* = \sigma(u^{-1})
\end{equation}
We clearly have that
\begin{eqnarray}
u^{**} = u, \qquad (uv)^* = v^* u^*, \qquad 1^* = 1, 
\label{propr-acao*} \\ 
uu^* = u \sigma(u^{-1}) = 1 \quad \Longleftrightarrow \quad u \in {\rm fix}(\sigma).
\label{isotropia-acao*}
\end{eqnarray}
By \eqref{isotropia-acao*} we have in particular that
\begin{equation}
k k^* = 1, \quad  \text{for }k \in K, 
\label{K-acao*}
\end{equation}
since $K \subseteq{\rm fix}(\sigma)$. By \eqref{propr-acao*} we have an action of $U$ on itself given by
\begin{equation}
\label{def-acao*}
u \cdot h = u h u^*
\end{equation}
which restricts to conjugation when $u \in K$.  

A Riemannian geometric symmetric space $S$ with geodesic reflection $s_p$ at $p$ can be (almost) embeeded in its isometry group by the Cartan map
\begin{equation}
\label{def-geom-embeed}
\eta: S \to I(S), \qquad q \mapsto s_q s_p^{-1},
\end{equation}
more precisely, this map is a covering onto its image.
Indeed, consider $U = I(S)_0$ with involution $\sigma(u) = s_p u s_p^{-1}$ and isotropy $K = U_p$. Since $s_{up} = u s_p u^{-1}$, it follows that
$$
\eta(u p) =  u s_p u^{-1} s_p^{-1} = u \sigma(u^{-1}) = u u^*  
$$
so that $\eta$ is $U$-equivariant, where on the right hand side $U$ acts on itself by the action \eqref{def-acao*}.  It follows that the image of $\eta$ is the orbit of $1$ under this action, which by \eqref{isotropia-acao*} has isotropy ${\rm fix}(\sigma)$.
Thus \eqref{def-geom-embeed} is a covering onto its image, with discrete fiber ${\rm fix}(\sigma)/K$.

\section{Weyl group action on $U/K$}
\label{sec:weyl-group-action}

In this section we extend to a maximal (flat) torus of a compact symmetric space $U/K$ the diagonalization result \eqref{eq-lemma-well-known} about the action of the Weyl group on $\t$.

Recall that, since the Weyl group $W$ of the symmetric space $U/K$ is given by infinitesimal data, it is given analytically as either one of the normalizers $M^* = N_K(\t)$ or $N_{K_0}(\t)$.  Also, an element $w \in W$ can be viewed both as a class $kM$ or as an orthogonal map $\Ad(k)|_\t$, where $k \in N_{K_0}(\t)$.
Recall the centralizer $M$ of $\t$ in $K$, the next result shows that this subgroup intersects each connected component of $K$.

\begin{proposition}
\label{propos:m-cc-k}
%
$K = K_0 M$.
\end{proposition}
\begin{proof}
Let $k \in K$, then $\t' = \Ad(k)\t$ is maximal abelian in $\s$. Thus, there exists $l \in K_0$ such that $\Ad(l) \t' = \Ad(lk)\t = \t$.  It follows that $lk \in M^*$, so that there exists $u \in N_{K_0}(\t)$ such that $\Ad(lk)H = \Ad(u)H$, for all $H \in \t$.
It follows that $u^{-1}lk \in M$, so that
$$
k = (l^{-1}u)(u^{-1}lk)
$$
where $l^{-1}u \in K_0$, as claimed.
\end{proof}

We first generalize \eqref{eq-lemma-well-known} to $T$, the connected subgroup of $U$ with Lie algebra $\t$. 

\begin{proposition}
\label{lemma:well-know}
Let $k \in K$ and $h \in T$ be such that $khk^{-1} \in T$, then
$$
khk^{-1} = w h w^{-1}
$$
for some $w \in W$.
%
\end{proposition}
\begin{proof}
Let $khk^{-1} = h' \in T$.  We have that
$$
T \subseteq Z_U(h')_0 = kZ_U(h)_0k^{-1} \supseteq kTk^{-1}.
$$
Furthermore, $Z_U(h')_0$ is $\sigma$-invariant. In fact, if $u \in Z_U(h')$ then, since $Z_U(h') = Z_U(h'^{-1})$, we have that
$$
\sigma(u)h' = \sigma(u)\sigma(h'^{-1}) = \sigma(u h'^{-1}) = \sigma(h'^{-1} u ) 
= \sigma(h'^{-1})\sigma(u)  = h' \sigma(u).
$$
The claim now follows since $Z_U(h') \supseteq \sigma( Z_U(h')_0 )$ is connected and contains the indentity.  Thus $\sigma$ restricts to an involution $\sigma'$ of $Z_U(h')_0$ which, intersected with ${\rm fix}(\sigma)_0 \subseteq K \subseteq{\rm fix}(\sigma)$, furnishes
$$
{\rm fix}(\sigma')_0 \subseteq K \cap Z_U(h')_0 \subseteq{\rm fix}(\sigma')
$$
where $Z_K(h')_0$ is contained in the middle subgroup. It follows that $(Z_U(h')_0, Z_K(h')_0)$ is a symmetric pair, where $T$ and $uTu^{-1}$ are maximal abelian symmetric in $Z_U(h')_0$. 
Thus, there exists $l \in Z_K(h')_0$ such that
$$
l kTk^{-1} l^{-1} = T
$$
so that $lk \in N_K(T)$ and thus $w = lkM \in W$ satisfies
$$
whw^{-1} = lkhk^{-1}l^{-1} = lh'l^{-1} = h'
$$
which proves the item.

\end{proof}

\subsection{Lifting to the universal cover of $U$}
\label{sec:lift}

We now lift some results from the compact connected Lie group $U$ to its universal cover $\wt\pi: \wt U \to U$, which is noncompact whenever $U$ is not semisimple.

Recall from \eqref{eq:centro} that the Lie algebra $\u$ of $\widetilde{U}$ splits as a direct sum of ideals $\u = \g \oplus \z(\u)$, where $\g$ is compact semi-simple and $\z(\u)$ the center of $\u$. It follows the direct product $\widetilde{U} = G \times Z$ where $G$ is the connected subgroup with Lie algebra $\g$, which is compact by a Theorem of Weyl (see Theorem V.1.3 of \cite{borel}), and $Z$ is a connected and central vector subgroup (see Corollary V.1.4 of \cite{borel}).


\begin{proposition}
\label{propos-fix-connected}
If $\widetilde{\sigma}$ is an automorphism of $\widetilde{U}$, then the set ${\rm fix}(\widetilde{\sigma})$ of its fixed points is connected.
\end{proposition}
\begin{proof}
We have that $\widetilde{\sigma}\left(G\right) = G$ and that $\widetilde{\sigma}\left(Z\right) = Z$, since these subgroups are characteristic. It follows that $\widetilde{\sigma} = \widetilde{\sigma}|_G \times \widetilde{\sigma}|_{Z}$, which implies that ${\rm fix}(\widetilde{\sigma}) = {\rm fix}\left(\widetilde{\sigma}|_G\right) \times {\rm fix}\left(\widetilde{\sigma}|_{Z}\right)$, which is connected, since ${\rm fix}\left(\widetilde{\sigma}|_{Z}\right)$ is obviously connected and since ${\rm fix}\left(\widetilde{\sigma}|_G\right)$ is also connected, by a Theorem of Borel (see Theorem V.3.3 of \cite{borel}).
\end{proof}


Recall from \eqref{eq:k-s-g} that 
\begin{equation}
\label{eq:u-g}
\u = (\k \cap \g) \oplus (\s \cap \g) \oplus 
(\k \cap \z(\u)) \oplus (\s \cap \z(\u))
\end{equation}

\begin{proposition}
We can always assume that $\k \subseteq\g$ and that $\z(\u) \subseteq\s$.
\end{proposition}
\begin{proof}
We have that
\[
\overline{\u} = (\k \cap \g) \oplus (\s \cap \g) \oplus (\s \cap \z(\u))
\]
is isomorphic to  $\u / (\k \cap \z(\u))$ and denote by $\overline{\sigma}$ the restriction of the involution $\sigma$ to $\overline{\u}$, and denote by an overline the image of the projection $\u \to \overline{\u}$ parallel to $\k \cap \z(\u)$. Thus
\[
\overline{\g} = \g, \qquad
\overline{\z(\u)} = \s \cap \z(\u) , \qquad
\overline{\k} = \k \cap \g, \qquad
\overline{\s} = \s
\]
and hence
\[
\overline{\u} = \overline{\g} \oplus \overline{\z(\u)} = \overline{\k} \oplus \overline{\s}
\]
so that
\[
\z(\overline{\u}) = \s \cap \z(\u) = \overline{\z(\u)}
\]
and
\begin{equation}
\label{eq-lifting-esp-sim}
{\rm fix}(\overline{\sigma}) = \k \cap \g = \overline{\k} = 
\overline{{\rm fix}(\sigma)}
\end{equation}

Now denote by $Z(U)$ the center of $U$,  let
\[
\overline{U} = U/(K \cap Z(U))
\]
$\overline{\sigma}$ the induced involution by $\sigma$ on $\overline{U}$, and denote by an overline the image of the projection $U \to \overline{U}$. We thus get the subgroups
$$
\overline{K} \subseteq \overline{{\rm fix}(\sigma)} 
\subseteq {\rm fix}(\overline{\sigma}) 
$$
which by \eqref{eq-lifting-esp-sim} have the same Lie algebra. It follows that $\overline{U}/\overline{K}$ is a compact Riemannian symmetric space such that 
\[
\frac{U}{K} \simeq \frac{\overline{U}}{\overline{K}}
\]
and such that $\overline{\k} \subseteq \overline{\g}$ and $\z(\overline{\u}) \subseteq \overline{\s}$, as claimed.
\end{proof}


Let $\widetilde{T}$ be the connected subgroup of $\widetilde{U}$ with Lie algebra $\t$. Denote by $\wt\sigma$ the lift of the involution $\sigma$ of $U$ to $\wt U$ so that
$\wt\pi \circ \wt\sigma = \sigma \circ \wt \pi$.

\begin{proposition}
\label{propos-polar-til}
Assume that $\k \subseteq\g$, and that $\z(\u) \subseteq\s$. Then we have that $\widetilde{T} = \left(\widetilde{T} \cap G \right) \times Z$, that $\widetilde{K} = {\rm fix}(\widetilde{\sigma}) \subseteq G$, and we have the polar decomposition
\[
\widetilde{U} = \widetilde{K}\widetilde{T}\widetilde{K}
\]
\end{proposition}
\begin{proof}
It is immediate that $\z(\u) \subseteq\t$, which implies that $\t = (\t \cap \g) \oplus \z(\u)$ and thus that $\widetilde{T} = \left(\widetilde{T} \cap G \right) \times Z$. Since $\wt U \to U$ is a covering, the Lie algebra of $\widetilde{K} = {\rm fix}(\widetilde{\sigma})$ is given by $\k$, which implies that  $\widetilde{K} \subseteq G$, since $\widetilde{K}$ is connected by Proposition \ref{propos-fix-connected}. Finally, since $G$ is compact and semi-simple and since $\t \cap \g \subseteq\s \cap \g$ is maximal abelian, it follows that
\[
G = \widetilde{K}\left(\widetilde{T} \cap G\right)\widetilde{K}
\]
Let $\pi : \widetilde{U} \to G$ be the natural projection homomorphism. It follows that $\pi(\widetilde{K}) = \widetilde{K}$, that $\pi(\widetilde{T}) = \widetilde{T} \cap G$, and thus that
\[
\pi(\widetilde{U}) = \pi\left(\widetilde{K}\widetilde{T}\widetilde{K}\right)
\]
This implies that, given $u \in \widetilde{U}$, there exist $k, l \in \widetilde{K}$, $h \in \widetilde{T}$, and $z \in Z$ such that $u = khlz$. Hence $u = khzl \in \widetilde{K}\widetilde{T}\widetilde{K}$, since $Z \subseteq\widetilde{T}$.
\end{proof}

We finish this section generalizing \eqref{eq-lemma-well-known} to $\wt T$.
Let $\wt W = N_{\wt K}(\wt T)$.

\begin{proposition}
\label{lemma:well-know-utilde}
Let $\wt k \in \wt K$ and $\wt h \in \wt T$ be such that $\wt k \wt h \wt k^{-1} \in \wt T$, then
$$
\wt k \wt h \wt k^{-1} = \wt w \wt h \wt w^{-1}
$$
for some $\wt w \in \wt W$. Furthermore, $\wt\pi(\wt W) \subseteq W$.
\end{proposition}
\begin{proof}
Denote by $\wt T' = \widetilde{T} \cap G$.  By the previous proposition, we have that $\widetilde{T} = \wt T' \times Z$ and $\widetilde{K} \subseteq G$. Restricting $\wt\sigma$ to $G$, we get the compact Riemannian symmetric pair $(G,\widetilde{K})$ with maximal torus $\wt T'$. 

Let $\wt k \wt h \wt k^{-1} = \wt t \in \widetilde{T}$.  Decompose $\wt h = \wt h' z$, $\wt t = \wt t' w$, where $\wt h', \wt t' \in \wt T'$ and $z, w \in Z$. Since $Z$ is contained in the center of $\wt U$ we get that
$$
\wt k \wt h' \wt k^{-1} z = \wt t' w
$$
Since $\widetilde{U} = G \times Z$, it follows that $\wt k \wt h' \wt k^{-1}= \wt t'$. Applying Proposition \ref{lemma:well-know} to $(G,\widetilde{K})$ we get $\wt w \in \wt W$ such that $\wt k \wt h' \wt k^{-1} = \wt w \wt h' \wt w^{-1} $. It follows that
$$
\wt k \wt h \wt k^{-1} = \wt k \wt h' \wt k^{-1} z = 
\wt w \wt h' \wt w^{-1} z = \wt w \wt h'z \wt w^{-1} = \wt w \wt h \wt w^{-1}
$$
as claimed. Furthermore, since $\wt\pi(\wt K) \subseteq K$ and $\wt\pi(\wt T) = T$, it follows that $\wt\pi( N_{\wt K}(\wt T) ) \subseteq N_K(T)$.
\end{proof}

\subsection{Lifting to the Klein universal cover of $U/K$}

We now realize the Klein universal cover from $\widetilde{U}$ to $S$: a good understanding of this covering allows us to transfer some results from $U/K$ to $U$ and vice versa. For $U$ compact semisimple this is standard material (see Chapter V of \cite{borel}) and here we extend this for $U$ compact.

Note that the Cartan covering \eqref{def-geom-embeed} factors through
\begin{equation}
\label{def-alg-embeed}
\eta: U/K \to U, \quad uK \mapsto uu^* \quad (u \in U),
\end{equation}
and the natural homomorphism from to $U$ to the isometry group $I(U/K)$,
where now $u^*$ is defined in terms of the involution $\sigma$ of $U$.  
This map is well defined by \eqref{K-acao*} and is automatically $U$-equivariant, where on the right hand side $U$ acts on itself by the action \eqref{def-acao*}. By the same reasoning as in Section \ref{sec:prelim}, \eqref{def-alg-embeed} is a covering onto its image, with fiber ${\rm fix}(\sigma)/K$, which we will also call the Cartan covering.
%

Denote by  $\wt\pi: \wt U \to U$ the universal cover.
Let $\wt K = {\rm fix}(\wt \sigma)$ which is connected by Proposition \ref{propos-fix-connected}, so that
\begin{equation}
\label{def:covering0}
\wt\pi(\wt K) = {\rm fix}(\sigma)_0 \subseteq K.
\end{equation}
and $\wt U/ \wt K$ is simply connected, by the homotopy exact sequence.  It follows that $\wt\pi: \wt U \to U$ descends to the universal cover of $U/K$, which we denote by the same symbol
\begin{equation}
\label{def:covering1}
\wt\pi: \wt U/ \wt K \to U / K, \qquad   u \wt K \mapsto \wt \pi( u ) K \quad (u \in \wt U).
\end{equation}
Since $\wt K = {\rm fix}(\wt \sigma)$, the Cartan covering \eqref{def-alg-embeed} becomes the embedding
$$
\wt \eta: \wt U/ \wt K \to \wt U, \quad u \wt K \mapsto uu^*  \quad (u \in \wt U)
$$
Denoting by $\wt S$ the image of $\wt\eta$ in $\wt U$, we have the following commutative diagram of universal covers
$$
\begin{tikzcd}
\widetilde{U}/\widetilde{K}
\arrow[rr, "\wt\eta", "\sim"{below}]
\arrow[rd, "\wt\pi"{below}]
& & \arrow{ld}{\lambda}  \widetilde S \, \subseteq\widetilde U \\
& S  &                                             
\end{tikzcd}
$$
which defines the {Klein universal cover $\lambda$.}


We now furnish a group theoretic description of the fibers of these coverings.
Let $\wt T$ be the connected subgroup of $\wt U$ with Lie algebra $\t$ and let
$$
\wt F = \wt T \cap Z(\wt U)
$$
Below we adapt ideas from Sections 6.10 of  \cite{HI12} and Theorem V.4.5 of \cite{borel} to prove the following.

\begin{theorem}
\label{thm:pi1}
We have that $\wt \eta$ restricts to a group isomorphism
\begin{equation}
\label{thm:pi1:eq1}
\wt \eta: N_{\wt U}(\wt K)/\wt K \to \wt F
\end{equation}
where the corresponding groups act effectively by isometries on the right of the respective universal cover.

Furthermore, $\wt \eta$ restricts to an isomorphism between the corresponding fibers $\wt\pi^{-1}(p)$, $\lambda^{-1}(p)$, which are subgroups isomorphic to $\pi_1(U/K)$.
\end{theorem}
\begin{proof}
Let $\wt\pi: \wt U \to U$, we claim that $\wt\pi^{-1}(K) \subseteq N_{\wt U}(\wt K)$.  Indeed, let $\wt\pi(u) \in K$. By \eqref{def:covering0} we have that
$$
\wt\pi(u \wt K u^{-1}) \subseteq\wt\pi(u) K \wt\pi(u)^{-1} = K,
$$
so that $u \wt K u^{-1}$ is connected with the same Lie algebra as $\wt K$, thus $u \wt K u^{-1} = \wt K$, as claimed.  
Going back to $\wt\pi: \wt U/\wt K \to U/K$, it follows that $\wt\pi^{-1}(p)  \leqslant N_{\wt U}(\wt K)/\wt K$.  By the commutativity of the diagram, we have that
\begin{equation}
\label{proof:thm-pi1:eq0}
\lambda^{-1}(p) = \wt \eta( \wt\pi^{-1}(p) ) \subseteq\wt \eta( N_{\wt U}(\wt K) /  \wt K ).
\end{equation}


Now we prove that
\begin{equation}
\label{proof:thm-pi1:eq1}
N_{\wt U}(\wt K) = (\wt T \cap N_{\wt U}(\wt K)) \wt K,
\end{equation}
where
\begin{equation}
\label{proof:thm-pi1:eq2}
\wt T \cap N_{\wt U}(\wt K) = \{ h \in \wt T:\, h^2 \in \wt F \}.
\end{equation}
For \eqref{proof:thm-pi1:eq1}, let $u \in N_{\wt U}(\wt K)$ and consider the polar decomposition $u = k h l$, where $k, l \in \wt K$ and $h \in \wt T$, given by Proposition \ref{propos-polar-til}. Since $u, k, l$ normalize $\wt K$, it follows that $h \in \wt T \cap N_{\wt U}(\wt K)$.  Furthermore, $k^{-1} u = hl$, where $k^{-1} u = u m$, for $m \in \wt K$, so that $u = h (lm^{-1})$, with $lm^{-1} \in \wt K$, as claimed.
For \eqref{proof:thm-pi1:eq2}, first note that
\begin{equation}
\label{proof:thm-pi1:eq3}
\wt T \cap Z_{\wt U}(\wt K) = \wt T \cap Z(\wt U),
\end{equation}
which follows from the polar decomposition of $u = khl  \in \wt U$, since $\wt h \in \wt T \cap Z_{\wt U}(\wt K)$ commutes with $h$, $k$ and $l$.
Now let $h \in \wt T$ and $k \in \wt K$, then $h^{-1} k h \in \wt K$ iff 
$$
h^{-1} k h = \wt \sigma( h^{-1} k h ) = h k h^{-1}
$$
iff $h^2 k h^{-2} = k$. It follows that
$$
\wt T \cap N_{\wt U}(\wt K) = \{ h \in \wt T:\, h^2 \in \wt T \cap Z_{\wt U}(\wt K)  \}
$$
and the claim follows from \eqref{proof:thm-pi1:eq3}.

For $h \in \wt T$ we have that $\wt \eta(h \wt K) = h^2$, so that \eqref{proof:thm-pi1:eq1} and \eqref{proof:thm-pi1:eq2} imply $\wt \eta( N_{\wt U}(\wt K) / \wt K ) \subseteq\wt F$.  
Since elements in $\wt T$ have square roots in $\wt T$, equality follows from \eqref{proof:thm-pi1:eq1} and \eqref{proof:thm-pi1:eq2}.  It follows that \eqref{thm:pi1:eq1} is a bijection. To prove that it is a group homomorphism, let $u \in \wt U$ and $n \in N_{\wt U}(\wt K)$.
Since $\wt \eta(n \wt K) \in \wt F \subseteq Z(\wt U)$, it follows by $\wt U$-equivariance that
$$
\wt \eta(un \wt K) = u \wt \eta(n \wt K) u^* = uu^* \wt \eta(n \wt K) = \wt \eta(u \wt K) \wt \eta(n \wt K), 
$$
which implies the claim.  From \eqref{proof:thm-pi1:eq0} it then follows that $\lambda^{-1}(p) \leqslant \wt F$ so that $\wt \eta$ restricts to a group isomorphism between $\wt\pi^{-1}(p)$ and 
$\lambda^{-1}(p)$, as claimed.

By the homotopy exact sequence $\pi_1(\wt U/ \wt K)$ is trivial, since $\wt K$ is connected. It follows that both $\wt \pi$ and $\lambda$ are coverings, so that $\pi_1(U/K)$ is isomorphic to both $\wt \pi^{-1}(p)$ and $\lambda^{-1}(p)$.
\end{proof}


\begin{remark}
Note that $h \in \wt T \cap \wt K$ iff $h^* = h^{-1} = h$, iff $h^2 = 1$.  Thus, a square root $a \in \wt T$ of $h \in \wt T$, that is, such that $h = a^2$, is defined mod $\wt K$.  Indeed, if $h = a^2 = b^2$, for $a, b \in \wt T$, then $(ab^{-1})^2 = 1$ so that $a = b$ mod $\wt K$.  Denote $\sqrt{h} = a$ mod $\wt K$, then by \eqref{proof:thm-pi1:eq1} and  \eqref{proof:thm-pi1:eq2} we have the isomorphism
$$
\wt F \to \wt T \cap N_{\wt U}(\wt K)/\wt K, \qquad f \mapsto \sqrt{f},
$$
which is the inverse to $\wt \eta$, through which $\wt F$ acts in the universal cover $\wt U/\wt K$.
\end{remark}


\subsection{Projecting back to $U/K$}
We finish this section by extending \eqref{eq-lemma-well-known} to the Weyl group $W$ of a compact symmetric space $U/K$.

\begin{theorem}
\label{theorem:well-know}
Let $M = U/K$ be a symmetric space with $U$ compact and  basepoint $p = K$. If $k \in K$ and $h \in T$ are such that $kh p \in T p$, then
$$kh p = wh p$$
for some $w \in W$ in the Weyl group of $U/K$.
\end{theorem}
\begin{proof}
Let $k h p = h' p$, for some $h' \in T$.  By Proposition \ref{propos:m-cc-k}, since $M \subseteq K$ centralizes $T$, we can assume w.l.o.g that $k \in K_0$.
We now use the notation of and results of Subsection \ref{sec:lift}.
Since $\wt K$ is connected, we have that $\wt\pi(\wt K) = K_0$. Let then $\wt k \in \wt K$, 
$\wt h \in \wt T$, 
$\wt h' \in \wt T$ such that 
$$
\wt \pi( \wt k ) = k, \quad \wt \pi( \wt h ) = h, \quad \wt \pi( \wt h' ) = h'.
$$
Denote the basepoint of $\wt U/\wt K$ by $\wt p = \wt K$, since 
$$
\wt \pi( \wt k \wt h \wt p ) = k h p = h' p = \wt \pi( \wt h' \wt p ) 
$$
it follows that $\wt k \wt h \wt p$  and $\wt h' \wt p$ in $\wt U/ \wt K$ are in the same fiber of $\wt \pi$. By Theorem \ref{thm:pi1}, their image by $\wt \eta$ are in the same fiber of the Klein universal cover $\lambda$, so that there exists $f \in \wt F \subseteq\wt T$ such that 
$$
\wt k \wt h^2  \wt k^{-1} = \wt \eta( \wt k \wt h \wt p )
= \wt \eta( h' \wt p )f = h'^2 f 
$$
where $h'^2 f  \in \wt T$. By Proposition \ref{lemma:well-know-utilde}, there exists $\wt w \in \wt W$, such that
$$
\wt k \wt h^2  \wt k^{-1} = \wt w \wt h^2  \wt w^{-1} 
$$
Thus
$$
\wt \eta( \wt k \wt h \wt p ) =\wt \eta( \wt w \wt h \wt p ) 
\quad \Longrightarrow \quad \wt k \wt h \wt p  = \wt w \wt h \wt p,
$$
since $\wt \eta$ is an embedding. Projecting back to $U/K$ we get that $khp = w hp$, for $w = \wt\pi(\wt w) \in W$, as claimed.
\end{proof}

\section{Inverse image of the exponential map}

In this section we will describe the inverse image of the exponential map $\exp_p: \s \to S$ given by \eqref{eq-exp-riemanniana-S}, as a disjoint union of focal orbits passing through $\t$, and will count their dimensions and connected components by inspecting data in $\t$.

Given $H \in \s$, recall that by \eqref{eq-diag} we can assume w.l.o.g.\ that $H \in \t$. Let $h = \exp(H)$ and $q = hp \in S$. Recall the focal orbit ${\mathcal F}(H) = K^q H$ given by \eqref{eq-def-focal-orbit}. Recall the lattice $\Gamma$ given by \eqref{def:lattice} and the centralizer $W^q$ of $H$ mod $\Gamma$ given by \eqref{def:centralizer}. 




\begin{theorem}
\label{thm:inverse-image}
Given $q = \exp(H)p$, we have that
$$
\exp_p^{-1}(q) = \bigcup_{\gamma \in \Gamma} {\mathcal F}(H + \gamma)
$$
where two such focal orbits intersect, and then coincide, if and only of their respective $H + \gamma$ lie in the same $W^q$-orbit.
\end{theorem}

%
%
%
%


\begin{proof}
First recall that $h = \exp(H)$ and note that the whole focal orbit
$$
{\mathcal F}(H + \gamma)
$$
exponentiates to $h p$ since, for $k \in K^q$, we have
$$
\exp( \Ad(k)(H + \gamma) )p = k\exp(H + \gamma)k^{-1}p = k\exp(H)p =kq = q 
$$
Reciprocally, let $\exp(X)p = h p$. There exists $k \in K$ such that $\Ad(k)X \in \t$ so that $\exp(\Ad(k)X)p = k h p \in Tp$.  
By Theorem \ref{theorem:well-know}, we have that $khp$ can be realized by 
$$
k h p = w h p
$$
for some $w \in W$, which we can consider belonging to $K$. It follows that $k^{-1} w \in K^q$.  In the above equation, replacing $hp$ by $\exp(X)p$ in the left hand side and by $\exp(H)p$ in the right hand side, we get that 
$$
\exp( \Ad(k) X ) p  = \exp( w H )p
$$
Since both $\Ad(k) X$ and $w H$ belong to $\t$, it follows that
$$
\Ad(k) X = w H + \gamma' = w( H + \gamma ) 
$$
for $\gamma = w^{-1} \gamma' \in \Gamma$, since $\Gamma$ is $W$-invariant. Then 
$$
X = \Ad(k^{-1})w(H + \gamma ) = 
\Ad(k^{-1}w) (H + \gamma )
$$
where we have seen that $k^{-1} w \in K^q$.  This proves thet $X$ is in the focal orbit ${\mathcal F}(H + \gamma)$, as claimed.

For the second part, if ${\mathcal F}(H + \gamma)$ and ${\mathcal F}(H + \gamma')$ intersect, then $H  + \gamma' = \Ad(u)(H + \gamma)$ for some $u \in K^q$.
By \eqref{eq-lemma-well-known}, it follows that $\Ad(u)(H + \gamma) = w(H + \gamma)$ for some $w \in W$, so that 
$$
H + \gamma' = w(H + \gamma) = wH + \gamma''
$$
where $\gamma'' = w \gamma \in \Gamma$, since $\Gamma$ is $W$-invariant.  Thus we get 
$$
wH = H + \gamma' - \gamma'',
$$
from which it follows that $w \in W^q$, as claimed.
Reciprocally, let $H + \gamma' = w(H + \gamma)$ for some $w \in W^q$.  Since $w$ can be realized by $u \in K$, we have that $H + \gamma' = \Ad(u)(H + \gamma)$. We get that $q = hp = u hp = uq$ and then $u \in K^q$. It follows that $H + \gamma' \in  \Ad(K^q)(H + \gamma) = {\mathcal F}(H + \gamma)$, as claimed.
\end{proof}

%

Consider the centralizers $K_H$ of $H$ and $K_h$ of $h$, so that
\begin{equation}
\label{eq:centralizers}
K_H \subseteq K_h \subseteq K^q
\end{equation}
Note that ${\mathcal F}(H)$ is diffeomorphic to $K^q/K_H$, where $K^q$ fixes the endpoints $p$ and $q$ while $K_H$ fixes pointwise the geodesic $\exp_p(tH)$, $t \in [0,1]$.
Recall from Equation \eqref{eq-algebra-centralizador} the Lie algebras of the centralizers
\begin{equation}
\label{eq-algebra-centralizador-2}
\k_H = \m \oplus \sum_{\alpha(H) = 0} \k_\alpha
\qquad
\k_h = \m \oplus \sum_{\alpha(H) \in 2\pi\Z} \k_\alpha
\end{equation}
where $\alpha$ ranges through the positive restricted roots.
Since the Cartan covering \eqref{def-alg-embeed} takes $k q = q$ to $kh^2k^{-1} = h^2$, it follows that the isotropy $K^q$ is contained in $K_{h^2}$ with the same Lie algebra
\begin{equation}
\label{eq-algebra-isotropia}
\k^q = \k_{h^2} = \m \oplus \sum_{\alpha(H) \in \pi\Z} \k_\alpha
\end{equation}
Note that the sum runs through the diagram hyperplanes that cross $H$.

Recall that elements of the Weyl group $W$ are represented by the centralizer $M_*$ of $\t$ in $K$, that is, $W = \Ad(M_*)|_{\t}$. Let
$$
M_*^q = M_* \cap K^q
$$
The next result shows that this subgroup intersects each connected component of $K^q$. Recall the subgroup $W^q_0$ generated by reflections around root hyperplanes paralell to the diagrams hyperplanes which cross $H$.

\begin{proposition}\label{propos:compon-conexas-k-q}
$K^q = K^q_0 M_*^q$ where
$$
W^q = \Ad(M_*^q)|_{\t} \qquad\text{and}\qquad  
W^q_0 = \Ad(K^q_0 \cap M_*^q)|_{\t}
$$
\end{proposition}
\begin{proof}
For the fisrt statement, let $q' = \exp(-H)p$, we claim that $K^q \subseteq K^{q'}$. Indeed, let $u \in K$ be such that $uq = q = \exp(H)p$. Applying the reflection of $U/K$ at $p$ we get
$$
s(u \exp(H)p) = \sigma(u)\exp(-H)p = u q' = s(\exp(H)p) = \exp(-H)p = q'
$$
Analogously we have $K^{q} \supseteq K^{q'}$, and thus $K^q = K^{q'}$.
Thus, by the covering $\eta: U/K \to U$ \eqref{def-alg-embeed}, it follows that $K^q$ is contained in
$$
\text{$U'=$ centralizer of $\eta(q) = h^2$ and $\eta(q') = h^{-2}$ in $U$}
$$
which is clearly closed and $\sigma$-invariant. Its Lie algebra then splits as 
$\u' = \k' \oplus \s'$, where $\k' = \u' \cap \k$ and $\s' = \u' \cap \s$.
It follows that $\k' = \k_{h^{\pm 2}} = \k_{h^2} = \k^q$, the Lie algebra $\k^q$ of $K^q$. Clearly, $\s'$ contains $\t$.  Let $u \in K^q$, then $\Ad(u)\t \subseteq\s'$ is a maximal torus. By Proposition \ref{propos:m-cc-k} applied to the symmetric pair $(U', K^q)$, there exists $v \in K^q_0$ such that $\Ad(vu)\t = \t$. Thus $vu = w \in K^q \cap M_* = M_*^q$ and $u = v^{-1} w$, with $v \in K^q_0$, which proves the first statement.

For the second statement, let $w \in W^q$ so that $wH = H + \gamma$, $\gamma \in \Gamma$. 
Let $v \in M_*$ represent $w$.
Thus, $vq = v\exp(H)v^{-1}p = \exp(wH)p = \exp(H + \gamma)p = \exp(H)p = q$, so that $v \in K^q$. 
Reciprocally, for $v \in M_*^q$ we have $\exp(\Ad(v)H)p = v\exp(H)p = vq = q = \exp(H)p$. Thus, $\Ad(v)H = H + \gamma$, $\gamma \in \Gamma$, so that $\Ad(v) \in W^q$,
as claimed.

For the third statement, since $\u'=\u_{h^{\pm 2}}$ it follows from 
\eqref{eq:centralizador-u} that
$$
(\u')_\C = \u'_0 \oplus \sum \{ \u'_\alpha:\, \alpha(2H) \in 2\pi\Z \} 
$$
Since $\t \subseteq \s'$, it follows that the restricted roots of $(\u', \k')$ are $\Pi' = \{ \alpha \in \Pi:\, \alpha(H) \in \pi\Z \}$, so that the Weyl group of $(\u', \k')$ is $W^q_0$. On the other hand, since $\k' = \k^q$, it follows that the Weyl group of $(\u', \k')$ is given by $\Ad$ of
$
K^q_0 \cap M_* = (K^q_0 \cap K^q) \cap M_* = K^q_0 \cap M_*^q
$,
as claimed.
\end{proof}


The dimension and connected components of a focal orbit $\mathcal{F}(H) = K^q H$ can be read in $\t$ as follows.

\begin{theorem}
\label{propos:cc-focal}
$\mathcal{F}(H)$ intersects $\t$ in the orbit $W^q H$ and its connected components are $K^q_0wH$, $w \in W^q$. They are in bijection with the quotient group $W^q/W^q_0$ where each connected component corresponds to a left $W^q_0$-orbit in $W^q H$ and is diffeomorphic to $K^q_0 H$, which has dimension
\begin{equation}
\label{eq:dim-focalorbit}
\dim {\mathcal F}(H)  = \sum m_\alpha
\end{equation}
where $m_\alpha$ is the multiplicity of the root $\alpha$ and the sum runs through the non-root hyperplanes of the Stiefel diagram that cross $H$. 
\end{theorem}
\begin{proof}
For the first statement, let $\Ad(k)H \in \t$, for $k \in K^q$. By \eqref{eq-lemma-well-known} there exists $w \in W$ such that $\Ad(k)H = wH$:
since the left hand side exponentiates to $q = \exp_p(H)$, it follows that $\exp_p(wH) = \exp_p(H)$. Thus, there exists $\gamma \in \Gamma$ such that 
$wH = H + \gamma$, so that $w \in W^q$, as claimed.

For the second statement, note that the centralizer $W_H$ is generated by reflections around root hyperplanes that cross $H$, so that $W_H \subseteq W^q_0$.  It follows from the previous Theorem that the connected components of $\mathcal{F}(H)$ are $K^q_0 w H$, $w \in W^q$. Thus they are in bijection with the double coset $W^q_0 \backslash W^q / W_H$. 
Since $W^q_0$ is normal, it follows that
$$
W^q_0 w W_H = w W^q_0 W_H = w W^q_0
$$
so that the double coset is in bijection with $W^q/W^q_0$.
For the last statement, let $w \in W^q$ be represente by $v \in  M_*^q$. Since conjugation preserves the identity component, it follows that each connected component of $\mathcal{F}(H)$ satisfies
$$
K^q_0 w H = v( v^{-1} K^q_0 v H ) = v( K^q_0 H )
$$
since conjugation by $v$ leaves invariant the connected component of the identity, thus is diffeomorphic to $K^q_0 H$, as claimed. Their dimension is then
$\dim \k^q/\k_H$, which by \eqref{eq-algebra-centralizador-2} and \eqref{eq-algebra-isotropia} is given by
$$
\dim \k^q/\k_H = \sum_{\alpha(H) \in \pi\Z - 0} m_\alpha 
$$
as claimed.
\end{proof}


\begin{remark}
For a compact Lie group $U$, the isotropy of the focal orbit $\mathcal{F}(H)$ is the Lie algebra centralizer $U_H$, which is connected, while the isotropy of the symmetric space $U$ is the Lie group centralizer $U_h$ (see Example \ref{ex:grupos}), thus the connected components of $\mathcal{F}(H)$ are in bijection with those of the isotropy $U_h$.
This does not hold in general for compact symmetric spaces (see Example \ref{ex:rank1}). What does hold in general is the above established bijection between the connected components of the focal orbit $\mathcal{F}(H)$ and the quotient group $W^q/W^q_0$.



\end{remark}

\section{Fundamental group}

In this section we generalize to compact symmetric spaces a well known description of the fundamental group of compact type symmetric spaces as a quotient of lattices: this is not immediate from the compact type case since the lattices in $\t$ do not need to be adapted to the splitting into center and its complement. We then use this result to investigate the connected components of $K^q$.

Let $\t_{\mathrm{reg}}$ be the complement in $\t$ of the Stiefel diagram and denote by $S_{\mathrm{reg}} = K \exp_p(\t_{\mathrm{reg}})$ be the regular elements of $S$. Note that the left action by $W$ and translation by $\Gamma$ leaves $\t_{\mathrm{reg}}$ invariant (see Proposition \ref{propos:gamma-inv}), thus so does the affine action of the semidirect product $\Gamma \rtimes W$. Also, $W$ acts on the right of $K/M$: it is the right action of $W = M_*/M$ on the principal bundle $K/M \to K/M_*$. 
It follows that $\Gamma \rtimes W$ acts in the right of $ \t_{\mathrm{reg}} \times K/M$ by the right $W$-action on $K/M$ and by the affine action on $\t_{\mathrm{reg}}$:
$$
(H, kM) \cdot (\gamma,w)= (\gamma + w^{-1}H, kwM)
$$
This action is free since $(\gamma + w^{-1}H, kwM) = (H, kM)$ implies $wM = M$, so that $w \in M$, thus $w^{-1}H = H$, which implies $\gamma + H = H$ so that $\gamma = 0$.

\begin{proposition}
\label{prop:pi1-reticulado}
The map
$$
\Psi: \t_{\mathrm{reg}} \times K/M \to S_{\mathrm{reg}},
\qquad
(H, kM) \mapsto k\exp_p(H)
$$
is the quotient map of the right $\Gamma \rtimes W$-action.
In particular, $\Psi$ is a covering of $S_{\mathrm{reg}}$ with fiber $\Gamma \rtimes W$.
\end{proposition}
\begin{proof}
Clearly, $\Psi(\gamma + w^{-1}H, k wM) = \Psi(H,kM)$.  Reciprocally, let $\Psi(H,kM) = \Psi(H',k'M)$, with $H, H' \in \t_{\mathrm{reg}}$. Then $k\exp_p(H) = k'\exp_p(H')$ so that
$$
q = \exp_p(H) = \exp_p(\Ad(k^{-1}k') H')
$$
where $q \in S_{\mathrm{reg}}$.
By Theorem \ref{thm:inverse-image} it follows that
$$
\Ad(k^{-1}k') H' = \Ad(u)(H + \gamma)
$$
for some $\gamma \in \Gamma$ and $u \in K^q$.
Since $q \in S_{\mathrm{reg}}$, it follows from 
equation \eqref{eq-algebra-isotropia} that $\k^q = \m$, so that $(K^q)_0 = M_0$. 
By Proposition \ref{propos:compon-conexas-k-q} it then follows that 
$$
K^q = M_0 M_*^q
$$
where $M_*^q$ acts in $\t$ as $W_q$ while $M$ centralizes $\t$, so that
$$
\Ad(u)(H + \gamma) = H + \gamma'
$$
for some $\gamma' \in \Gamma$, by the definition $W^q$.
It follows that
$$
\Ad(k^{-1}k') H' = H + \gamma'
$$
so that by \eqref{eq-lemma-well-known} we have $\Ad(k^{-1}k') H' = w H'$, for some $w \in W$. Then
$$
wH' = H + \gamma' \quad \Longrightarrow \quad H' = \gamma'' + w^{-1}H
$$
where $\gamma'' = w^{-1} \gamma' \in \Gamma$. Now $\Ad(w^{-1}k^{-1}k')H' = H'$ with $H'$ regular implies that $w^{-1}k^{-1}k' \in M$, so that $k'M = kwM$, and then
$$
(H', k'M) = (H, kM) \cdot (\gamma'', w)
$$
as claimed.
\end{proof}

The determination of the fundamental group of compact symmetric spaces follows using the same strategy which is used in the proof of the analogous result for compact Lie groups in Chapter 13 of \cite{Hall} apart from the first and, specially, the last item, which we addressed above. 

\begin{theorem}
\label{thm:pi1-reticulado}
Is $S$ is a compact symmetric space, then 
\[
\pi_1(S) = \Gamma/\Gamma_0
\]
\end{theorem}
\begin{proof}
The proof is given by the following items.  

 
 
\begin{enumerate}[1)]
    \item By a standard curve-shortening process, one shows that any nontrivial class in $\pi_1(S)$ can be represented by a closed geodesic through $p$ (see Lemma 8.3.10 of \cite{wolf}), say $\gamma(t) = \exp_p(tX)$, $X \in \s$, with $\gamma(1) = \exp_p(X) = p$.  By \eqref{eq-diag} there exists $k \in K_0$ such that $kX = H \in \t$. Note that $\exp_p(H) = k\exp_p(X) = kp = p$, thus $H \in \Gamma$. Since $K_0$ is connected, it follows that $\gamma(t)$ is homotopic the the closed geodesic $\exp_p(tH)$ on $Tp$. Thus, the inclusion
    $$\pi_1(Tp) \to \pi_1(S)$$
    is surjective.
    Note that \cite{Hall} uses the simply connectivity of the full flag manifold of $K$ to prove this (Proposition 13.37 of \cite{Hall}), but here the corresponding object is the full flag manifold $K/M$ of $S$, which is not always simply connected. 

    \item Since $Tp$ is a compact abelian Lie group, $\exp_p: \t \to Tp$ is a homomorphism with kernel $\Gamma$, which induces an isomorphism $\t/\Gamma \to Tp$. It follows that the inclusion of the closed geodesics $\Gamma \to \pi_1(Tp)$ induces an isomorphism.
    
    \item By the standard construction given in the proof of Proposition \ref{propos:gamma-inv} we can realize each closed geodesic of $S$ corresponding to a coroot vector as the equator of a corresponding 3-sphere immersed in $U$ and acting in $S$. By a standard argument (entirely analogous to {\it Proof of Theorem 13.17, one direction}, p.382 of \cite{Hall}), it follows that every closed geodesic that comes from the fundamental lattice $\Gamma_0$ is nullhomotopic in $S$. 
    
    \item By a standard dimension argument (entirely analogous to the Proof of Theorem 13.27, p.386 of \cite{Hall}) it follows that the fundamental group of $S$ equals the fundamental group of $S_{\mathrm{reg}}$.
    
    \item The proof that every loop that does not come from $\Gamma_0$ is not nullhomotopic ({\it Proof of Theorem 13.17, the other direction}, p.398 of \cite{Hall}) depends on the construction of a covering of $S_{\mathrm{reg}}$. \cite{Hall} obtains the universal covering, but a covering is enough for the proof. This construction is the only step whose proof fully departs from the compact group case of \cite{Hall} and it is provided by Proposition \ref{prop:pi1-reticulado}.
\end{enumerate}
\end{proof}

\begin{corollary}
\label{corol:simply-connected}
If $S$ is simply connected then
$$
\Gamma = \Gamma_0 \qquad {\mathcal D} = {\mathcal D}_0 \qquad W^q = W^q_0
$$ 
In particular, all focal orbits are connected.
\end{corollary}
\begin{proof}
Since $\pi_1(S)$ is trivial, Theorem \ref{thm:pi1-reticulado} implies that $\Gamma = \Gamma_0$, then Theorem \ref{thm-steinberg} implies that ${\mathcal D} = {\mathcal D}_0$ and $W^q = W^q_0$. The connectedness of the focal orbits then follows from Theorem \ref{propos:cc-focal}.
\end{proof}

\begin{proposition}
Each connected component of a focal orbit of corresponds to a distinct homotopy class of curves between $p$ and $q$ by the correspondence
\begin{eqnarray}
\label{eq-injection-pi1}
\pi_0(\mathcal{F}(H)) = W^q/W^q_0 &\,\,\hookrightarrow\,\,& \pi_1(S) = \Gamma/\Gamma_0 \\ 
w \text{ \rm mod } W^q_0 & \,\,\mapsto\,\,& wH - H \text{ \rm mod } \Gamma_0 
\nonumber
\end{eqnarray}
\end{proposition}
\begin{proof}
Fixing the geodesic $\gamma_H(t) = \exp_p(tH)$ from $p$ to $q = \exp_p(H)$, $H \in \t$, gives us the map
$$
\pi_{p,q}(S) \to \pi_1(S), \qquad [\gamma] \mapsto [\gamma * \gamma_{H}^{-1}]
$$
between the homotopy classes of curves between $p$ and $q$ and the fundamental group of loops at $p$. This map is a bijection with inverse  
$\pi_1(S) \to \pi_{p,q}(S)$, $[\gamma] \mapsto [\gamma * \gamma_{H}]$.
Restricting this to the geodesics corresponding to 
a connected component of a focal orbit $\mathcal{F}(H)$, by Theorem \ref{propos:cc-focal} we get the injection \eqref{eq-injection-pi1}, as claimed.
\end{proof}

Using that $W_0^q$ is the centralizer of $H$ modulo $\Gamma_0$ (see remark after Theorem \ref{thm-steinberg}), one can prove directly, by taking quotients, that the map \eqref{eq-injection-pi1} is an injection.

\section{Conjugate and cut loci}

Crittenden \cite{crittenden} and Sakai \cite{sakai} characterized, respectively, the first conjugate locus and the cut locus of compact symmetric spaces showing that they are determined by a maximal flat torus. In this section we use our results to reobtain their results, with short and independent proofs.  

Recall that a tangent vector $X \in T_pS$ is a (first tangent) conjugate point at $p$ when $tX$ has null index for $0 \leq t < 1$ and positive index for $t=1$, where the index of $Y \in T_pS$ along the geodesic $c_Y(t) = \exp_p(tY)$ is the dimension of the kernel of the differential $d_Y \exp_p: T_p S \to T_{\exp_p(Y)} S$. 
The (first tangent) conjugate locus at $p$ is the set of first tangent conjugate points in $T_pS$.
Conjugate points are related to variations by geodesic as follows: $X \in T_pS$ is a tangent conjugate point at $p$ if and only if there exists a nontrivial variation of geodesics around $c_X$ with fixed start point $p$ at $t=0$ and, {\em up to first order}, fixed end point $q=\exp_p(X)$ at $t=1$. 
Now, recall that a tangent vector $X \in T_pS$ is a (tangent) cut point at $p$ when the geodesic $c_X(t) = \exp_p(tX)$ is minimizing for $0 \leq t \leq 1$ and ceases to be minimizing for $t>1$. The (tangent) cut locus at $p$ is the set of tangent cut points in $T_pS$.

The following result is well known and shows that tangent conjugate points of a compact symmetric space $S = U/K$ are intimately related to the focal orbits that appear in Theorem \ref{thm:inverse-image}.

\begin{lemma}
\label{propos:indice}
The index of $H \in \t$ along the geodesic $c_H(t) = \exp_p(tH)$ equals
$\dim {\mathcal F}(H)$.  Furthermore, $H$ is a tangent conjugate point of the geodesic $c_H$ iff there exists a nontrivial variation of geodesics around $c_H$ with fixed start point $p$ and fixed end point $q=\exp_p(H)$.  Note that such a variation corresponds to a curve in ${\mathcal F}(H)$ that passes through $H$.
\end{lemma}
\begin{proof}
By Theorems IV.4.1 and Lemma VII.2.9 of \cite{helgason}, the eigenvalues of $d_H \exp_p(H)$ are given by $\sin(\alpha(H))/\alpha(H)$ with multiplicity $\dim \s_\alpha = m_\alpha$,  for $\alpha \in \Pi^+$, thus is zero iff $\alpha(H) \in  \pi\Z - 0$. 
This, together with equation \eqref{eq:dim-focalorbit} proves the first part.

For the the second part, if $H$ is tangent conjugate then, by the previous paragraph, $\dim {\mathcal F}(H) > 0$, so that there exists a nontrivial variation of geodesics around $c_H$ with fixed endpoints $p$ and $q$, which corresponds to a nontrivial curve in ${\mathcal F}(H)$ passing through $H$. The reciprocal is immediate from the second equivalent definition of conjugate point.
\end{proof}

For the cut points of a compact symmetric space $S = U/K$, we first investigate the submanifold geometry of a maximal flat torus $Tp$. It is immediate that $Tp$ is totally geodesic in $S$, moreover, we have the following.

\begin{proposition}
\label{lema-metrica}
$Tp$ is geodesically convex in $S$, that is, the Riemannian distance of $S$ induced in $Tp$ coincides with its intrinsic Riemannian distance.
\end{proposition}
\begin{proof}
Denote by $d_S$, $d_T$ the Riemmanian distances of $S$ and $Tp$, respectively.
For $x, y \in Tp$ it is immediate that $d_S(x,y) \leq d_T(x,y)$.  For the converse inequality, since $T$ acts transitively by isometries on $T$ and $S$, we can assume that $x=hp$, $y=p$, where $h \in T$.
Let $X$ of minimal norm in $\s$ such that $hp = \exp_p(X)$, then $d_S(hp,p) = |X|$, since $\exp_p(tX)$, $0 \leq t \leq 1$, is a geodesic of length $|X|$ between $p$ and $\exp_p(X)$ in $S$.
Now let $H$ of minimal norm in $\t$ such that $hp = \exp_p(H)$, then $d_T(p,hp) = |H|$, arguing as before, now in $Tp$.
Since $\exp_p(X) = \exp_p(H) = hp$, it follows from Theorem \ref{thm:inverse-image} that $X \in {\mathcal F}(H + \gamma)$, for some $\gamma \in \Gamma$.  Hence $|X| = |H + \gamma| \geq |H|$, by the minimality of $|H|$. Thus $d_S(hp,p) \geq d_T(hp,p)$, as claimed.
\end{proof}

Consider the Dirichilet domain ${\mathcal D}$ of the lattice $\Gamma$ and the Dirichilet domain ${\mathcal D}_0$ of the fundamental lattice $\Gamma_0$ (see Section \ref{Weyl-group}).

\begin{theorem}
\label{teo:loci}
Let $S = U/K$ be a compact symmetric space, then
\begin{enumerate}[i)]
 

\item (Sakai) The cut locus of $S$ is the $K$ orbit of $\partial {\mathcal D}$.

\item (Crittenden) The conjugate locus of $S$ is the $K$ orbit of $\partial {\mathcal D}_0$. 
\end{enumerate}
Furthermore, $S$ is simply connected iff its cut locus equals its first conjugate locus.
\end{theorem}
\begin{proof}
We will use freely the results of Proposition \ref{prop:reticulado}.
We first claim that the cut locus of $Tp$ is the boundary of ${\mathcal D}$. Indeed, recall that $\ov{\mathcal D}$ is the set of vectors of $\t$ with minimal norm among its $\Gamma$ equivalents. 
Let $H$ be in the tangent cut locus of $T$, then $tH \in \ov{\mathcal D}$ for $0 \leq t \leq 1$, and $tH \not\in \ov{\mathcal D}$ for $t > 1$ so that $H \in \partial {\mathcal D}$.
Reciprocally, let $H$ be in the boundary of ${\mathcal D}$. Since ${\mathcal D}$ is a convex neighbourhood of the origin, it follows that $tH \in \ov{\mathcal D}$ is minimal for $0 \leq t \leq 1$. For $t=1$ we have $H \in \partial {\mathcal D}$ so that there exists $\gamma \neq 0$ in $\Gamma$ such that $|H + \gamma| = |H|$. Thus $Y = H + \gamma$ is a nontrivial focal equivalent to $H$ and by Lemma \ref{lema:reticulado} it follows that $(1+\epsilon)H$ has norm greater than $Y + \epsilon H$, for $\epsilon > 0$, while $\exp((1+\epsilon)H) = \exp(Y + \epsilon H)$. It follows that $\exp(tH)$ is not a minimizing geodesic for $t > 1$, which proves our claim.

For item (i), from Proposition \ref{lema-metrica} it follows that the cut locus of $Tp$ is contained in the cut locus of $S$, and so its adjoint orbit, since $\Ad(K)$ acts by isometries in $\s$.  By using that $\Ad(K)$ conjugates every $X \in \s$ to some $H \in \t$, it follows from Proposition \ref{lema-metrica} that the cut locus of $S$ is contained in the adjoint orbit of the cut locus of $Tp$, which proves item (i).

We now claim that the conjugate locus of $Tp$ is the boundary of ${\mathcal D}_0$. 
Indeed, recall that $\ov{{\mathcal D}_0} = \{ H \in \t: \, \alpha(H) \leq \pi,\, \alpha \in \Pi \}$. Thus its boundary $\partial {\mathcal D}_0$ is contained in the union of hyperplanes $\alpha(H) = \pi,\, \alpha \in \Pi$. The claim now follows since ${\mathcal D}_0$ consists of regular elements and ${\mathcal D}_0$ is a convex neighbourhood of the origin so that, by Lemma \ref{propos:indice} the first conjugate point of any ray $tH$ occurs when it hits the boundary of ${\mathcal D}_0$.

For item (ii), it is obvious that the conjugate locus of $Tp$ is contained in the tangent conjugate locus of $S$, and so its adjoint orbit, since $\Ad(K)$ acts by isometries in $\s$.  By using that $\Ad(K)$ conjugates every $X \in \s$ to some $H \in \t$ and that $\Ad(K)$ acts by isometries, it follows that the tangent conjugate locus of $S$ is contained in the $\Ad(K)$ orbit of the tangent conjugate locus of $T$, which proves the first claim.
For the second claim, by the previous items, the cut locus of $S$ equals its first conjugate locus iff the cut locus of $Tp$ equals its first conjugate locus iff ${\mathcal D} = {\mathcal D}_0$ since both are open convex sets whose closure is the convex closure of its boundary points.  If $\Gamma = \Gamma_0$ then it is immediate that ${\mathcal D} = {\mathcal D}_0$.  Reciprocally, if ${\mathcal D} = {\mathcal D}_0$ we claim that 
$$
\min_{\gamma_0 \in \Gamma_0} |H + \gamma_0| = 
\min_{\gamma \in \Gamma} |H + \gamma|
$$ 
Indeed, we have that $\min_{\gamma_0 \in \Gamma_0} |H + \gamma_0|$ is attained at a $H + \gamma_0 \in \ov{\mathcal D}_0$, for some $\gamma_0 \in \Gamma_0$. Since
${\mathcal D}_0 = {\mathcal D}$, we have that $|H + \gamma_0|$ equals
$\min_{\gamma \in \Gamma} |(H + \gamma_0) + \gamma|$.  Since $\Gamma_0 \subseteq\Gamma$, this equals $\min_{\gamma \in \Gamma} |H + \gamma|$, which proves the claim.
Taking $H = \gamma \in \Gamma$ it follows that $\min_{\gamma_0 \in \Gamma_0} | \gamma + \gamma_0| = 0$ so that $\gamma \in \Gamma_0$ and thus $\Gamma \subseteq\Gamma_0$ so that $\Gamma = \Gamma_0$.
By Theorem \ref{thm:pi1-reticulado}, $S$ is simply connected iff $\Gamma = \Gamma_0$, which proves our claim.
\end{proof}

\end{document}